\documentclass[10pt]{amsart}
\usepackage{amsmath, amssymb, latexsym}
\usepackage{mathrsfs}
\usepackage[poly,all]{xy}
\usepackage{texdraw}

\newtheorem{Thm}{Theorem}[section]
\newtheorem{Prop}[Thm]{Proposition}
\newtheorem{Cor}[Thm]{Corollary}
\newtheorem{Lem}[Thm]{Lemma}

\theoremstyle{definition}
\newtheorem{Def}[Thm]{Definition}
\newtheorem{Exa}[Thm]{Example}

\numberwithin{equation}{section}

\newcommand{\Z}{\mathbb{Z}}
\newcommand{\ZZ}{\mathbb{Z}_{\geq 0}}

\newcommand{\C}{\mathbb{C}}

\newcommand{\g}{\mathfrak{g}}
\newcommand{\wt}{{\rm wt}}
\newcommand{\qin}{{\rm in}}
\newcommand{\qout}{{\rm out}}
\newcommand{\tr}{{\rm tr}}
\newcommand{\Hom}{{\rm Hom}}
\newcommand{\Aut}{{\rm Aut}}
\newcommand{\End}{{\rm End}}
\newcommand{\Irr}{{\rm Irr}}
\newcommand{\Color}{{\rm Color}}
\newcommand{\cl}{{\rm cl}}
\newcommand{\Span}{{\rm span}}
\newcommand{\im}{{\rm im}}
\newcommand{\Ht}{{\rm ht}}
\newcommand{\adjoint}{B^{\rm ad}}

\begin{document}

\title[Quiver Varieties and Path Realizations Arising From Adjoint Crystals of type $A_n^{(1)}$]
{Quiver Varieties and Path Realizations \\ arising from Adjoint
Crystals of type $A_n^{(1)}$}
\author[Seok-Jin Kang]{Seok-Jin Kang$^1$}
\thanks{$^1$ This research was supported by KRF Grant \# 2007-341-C00001.}
\address{Department of Mathematical Sciences and Research Institute of Mathematics,
Seoul National University, Gwanak-ro 599, 
Gwanak-gu, Seoul 151-747, Korea} \email{sjkang@math.snu.ac.kr}
\author[Euiyong Park]{Euiyong Park$^{1,2}$}
\thanks{$^2$ This research was supported by BK21 Mathematical Sciences Division.}

\thanks{$^\dag$ This is accepted for publication by the Transactions of the American Mathematical Society.}

\address{Department of Mathematical Sciences, Seoul National University,
Gwanak-ro 599, 
Gwanak-gu, Seoul 151-747, Korea}
\email{pwy@snu.ac.kr}


\begin{abstract}
Let $B(\Lambda_0)$ be the level 1 highest weight crystal of the
quantum affine algebra $U_q(A_n^{(1)})$. We construct an explicit
crystal isomorphism between the geometric realization
$\mathbb{B}(\Lambda_0)$ of $B(\Lambda_0)$ via quiver varieties and
the path realization ${\mathcal P}^{\rm ad}(\Lambda_0)$ of
$B(\Lambda_0)$ arising from the adjoint crystal $\adjoint$.

\end{abstract}

\maketitle


\section*{Introduction}

The theory of {\it perfect crystals} developed in \cite{KKMMNN91}
has a lot of important and interesting applications to the
representation theory of quantum affine algebras and the theory of
vertex models in mathematical physics. In particular, the crystal
$B(\lambda)$ of an integrable highest weight module over a quantum
affine algebra can be realized as the crystal ${\mathcal
P}^{B}(\lambda)$ consisting of $\lambda$-paths arising
from a given perfect crystal $B$. In \cite{BFKL06},
Benkart, Frenkel, Kang and Lee gave a uniform construction of
level 1 perfect crystals for all quantum affine algebras.
These perfect crystals are called the {\it adjoint
crystals} because, when forgetting 0-arrows, they coincide with
the direct sums of the trivial crystals and the crystals of
adjoint or little adjoint representations of finite dimensional
simple Lie algebras.

On the other hand, for a symmetric Kac-Moody algebra $\g$, Lusztig
gave a geometric construction of $U_q^{-}(\g)$ in terms of
perverse sheaves on quiver varieties and introduced the notion of
{\it canonical basis} which yields natural bases for all
integrable highest weight modules as well \cite{L90, L91}. In
\cite{KS97}, Kashiwara and Saito defined a crystal structure on
the set ${\mathbb B}(\infty)$ of irreducible components of
Lusztig's quiver varieties and showed that ${\mathbb B}(\infty)$
is isomorphic to the crystal $B(\infty)$ of $U_q^{-}(\g)$.
Moreover, in \cite{N94, N98}, Nakajima defined a new family of
quiver varieties associated with a dominant integral weight
$\lambda$ and gave a geometric realization of the integrable
highest weight $\g$-module $V(\lambda)$.
In \cite{St02}, Saito defined a crystal
structure on the set ${\mathbb B}(\lambda)$ of irreducible
components of certain Lagrangian subvarieties of Nakajima's quiver
varieties, and showed that ${\mathbb B}(\lambda)$ is isomorphic to
the crystal $B(\lambda)$ of $V(\lambda)$.

Therefore, for quantum affine algebras, it is natural to
investigate the crystal isomorphism between the geometric
realization ${\mathbb B}(\lambda)$ and the path realization
${\mathcal P}^{B}(\lambda)$ for various perfect crystals $B$. In
this paper, we will focus on the level 1 highest weight crystal
$B(\Lambda_0)$ of the quantum affine algebra $U_q(A_n^{(1)})$, and
construct an explicit crystal isomorphism between ${\mathbb
B}(\Lambda_0)$ and ${\mathcal P}^{\rm ad}(\Lambda_0)$, where
${\mathcal P}^{\rm ad}(\Lambda_0)$ is the path realization arising
from the adjoint crystal $\adjoint$. We will also give a geometric
interpretation of the {\it fundamental isomorphism theorem for
perfect crystals}: $B(\Lambda_0) \cong B(\Lambda_0) \otimes
\adjoint$. One of the key ingredients of our construction is the
explicit 1-1 correspondence between ${\mathbb B}(\Lambda_0)$ and
${\mathcal Y}(\Lambda_0)$ discovered in \cite{FS03, Sv06}, where
${\mathcal Y}(\Lambda_0)$ is the crystal consisting of
$(n+1)$-reduced Young diagrams. We hope our construction will
provide a new insight toward the understanding of the connection
between ${\mathbb B}(\Lambda_0)$ and ${\mathcal
P}^{\rm ad}(\Lambda_0)$ for {\it all} quantum affine algebras.

\vskip 3mm

\noindent
{\bf Acknowledgments.} The authors wish to express their
sincere gratitude to National Institute for Mathematical Sciences in
Daejon, Korea, for their generous support during the special program (Thematic Program TP0902)
on representation theory in 2009.

 \vskip 3em

\section{The quantum affine algebra $U_q(A_n^{(1)})$}

Let $I=\Z/(n+1)\Z$ be the index set. The {\it affine Cartan
datum of $A_n^{(1)}$-type} consists of

(i) the affine Cartan matrix $$A=(a_{ij})_{i,j\in I} = \left(
\begin{matrix} 2 & -1 & 0 & \cdots & -1 \\ -1 & 2 & -1 & \cdots & 0 \\
\vdots &  & \ddots & & \vdots \\ 0 & \cdots & -1 & 2 & -1 \\ -1 & 0
& \cdots & -1 & 2 \end{matrix} \right),$$

(ii) dual weight lattice $P^{\vee}= \bigoplus_{i=0}^{n}\Z h_i \oplus
\Z d$,

(iii) affine weight lattice $P = \bigoplus_{i=0}^n \Z \Lambda_i
\oplus \Z \delta \subset \mathfrak{h}^*$, where
$$\mathfrak{h} = \C
\otimes P^{\vee},\ \Lambda_i(h_j) = \delta_{ij},\ \Lambda_i(d)=0,\
\delta(h_i)=0,\ \delta(d)=1\ (i,j\in I),$$

(iv) the set of simple coroots $\Pi^{\vee} = \{ h_i|\ i\in I \}$,

(v) the set of simple roots $\Pi = \{ \alpha_i|\ i\in I \}$ given by
$$\alpha_j(h_i) = a_{ij},\ \alpha_j(d)=\delta_{0,j}\ (i,j\in
I).$$ The free abelian group $Q = \bigoplus_{i=0}^n \Z \alpha_i$
is called the {\it root lattice} and the semigroup $Q^+ =
\sum_{i=0}^n \ZZ \alpha_i$ is called the {\it positive root
lattice}. For $\alpha = \sum_{i\in I} k_i \alpha_i \in Q^+$,
the number $\Ht(\alpha) = \sum_{i\in I} k_i$ is called the {\it height} of $\alpha$.
For $\lambda, \mu \in \mathfrak{h}^*$, we define
$\lambda \ge \mu$ if and only if $\lambda - \mu \in Q^{+}$. The
elements in $ P^+ = \{ \lambda \in P|\ \lambda(h_i) \ge 0,\ i\in I
\} $ are called the {\it dominant integral weights}. Note that the
minimal {\it imaginary root} is given by $\delta = \alpha_0 +
\alpha_1 + \cdots + \alpha_n \in Q^{+}$. The element $c=h_0 + h_1
+ \cdots + h_n \in P^{\vee}$ is called the {\it canonical central
element}.

Given $n \in \Z$ and any symbol $q$, we define
$$ [n]_q = \frac{q^n - q^{-n}}{q-q^{-1}}$$
and set $[0]_q ! = 1$, $[n]_q! = [n]_q[n-1]_q\cdots [1]_q $. For
$m\ge n \ge 0 $, let
$$ \left[
     \begin{array}{c}
       m \\
       n \\
     \end{array}
   \right]_q
   =\frac{[m]_q !}{[n]_q! [m-n]_q !}.
 $$
\begin{Def}
The {\it quantum affine algebra} $U_q(\g)=U_q(A_n^{(1)})$ is an
associative algebra over $\C(q)$ with $1$ generated by $e_i,\ f_i\
( i \in I)$ and $q^{h}\ (h \in P^{\vee})$ satisfying the defining
relations:
\begin{enumerate}
\item[(a)] $q^0=1,\ q^h q^{h'} = q^{h+h'}$ for $h,h' \in
P^{\vee}$,
\item[(b)] $q^he_iq^{-h}=q^{\alpha_i(h)}e_i$ for $h \in P^{\vee}$,
\item[(c)] $q^hf_iq^{-h}=q^{-\alpha_i(h)}f_i$ for $h \in P^{\vee}$,
\item[(d)] $e_if_j- f_je_i = \delta_{ij}(q^{h_i} - q^{-h_i})/(q - q^{-1})$ for $i,j \in I$,
\item[(e)] $\sum_{k=0}^{1-a_{ij}}(-1)^k
{1-a_{ij} \brack k}_q\  e_i^{1-a_{ij}-k}e_j e_i^k =0$ for $i \ne j$,
\item[(f)] $\sum_{k=0}^{1-a_{ij}}(-1)^k
{1-a_{ij} \brack k}_q\  f_i^{1-a_{ij}-k}f_j f_i^k =0$ for $i \ne j$.
\end{enumerate}
\end{Def}

The definition of {\it category $O_{\rm int}^q$}, {\it Kashiwara
operators} and {\it crystal bases} can be found in \cite{K91, HK02},
etc. It was shown in \cite{K91} that every $U_q(\g)$-module in the
category $O_{\rm int}^q$ has a crystal basis. The notion of {\it
abstract crystals} was introduced in \cite{K93}. For convenience, we
recall some of the basic definitions and properties of abstract
crystals.

\begin{Def} An {\it abstract crystal} associated with the Cartan datum $(A, \Pi, \Pi^{\vee}, P, P^{\vee})$ is a set
$B$ together with the maps $\wt: B \to P,\ \tilde{e}_i, \tilde{f}_i: B \to
B \sqcup \{0\},$ and $\varepsilon_i, \varphi_i: B \to \Z \cup \{ -\infty \}\ (i\in I)$ satisfying
the following properties:
\begin{enumerate}
\item[(a)] $\varphi_i(b) = \varepsilon_i(b) + \langle h_i, \wt(b) \rangle$ for all $i\in I$,
\item[(b)] $\wt(\tilde e_i b) = \wt(b)+\alpha_i$ if $\tilde{e}_ib \in B$,
\item[(c)] $\wt(\tilde f_i b) = \wt(b)-\alpha_i$ if $\tilde{f}_ib \in B$,
\item[(d)] $\varepsilon_i(\tilde e_i b)= \varepsilon_i(b)-1,\ \varphi_i(\tilde e_i b) = \varphi_i(b)+1$ if $\tilde e_i b \in B$,
\item[(e)] $\varepsilon_i(\tilde f_i b)= \varepsilon_i(b)+1,\ \varphi_i(\tilde f_i b) = \varphi_i(b)-1$ if $\tilde f_i b \in B$,
\item[(f)] $\tilde f_i b = b'$ if and only if $b=\tilde e_i b'$ for $b,b'\in B,\ i\in I$,
\item[(g)] if $\varphi_i(b)=- \infty$ for $b \in B$, then $\tilde e_i b = \tilde f_i b = 0$.
\end{enumerate}
\end{Def}
We often say that $B$ is a {\it $U_q(\g)$-crystal}. We denote
$B_{\lambda} = \{ b \in B|\ \wt(b)=\lambda \}$ so that $B =
\bigsqcup_{\lambda \in P} B_\lambda$.

\begin{Exa} \
\begin{enumerate}

\item Let $(L,B)$ be a crystal basis of $M \in O_{\rm int}^q$. Then $B$ has a crystal structure,
where the maps
$\varepsilon_i,\ \varphi_i$ are given by
$$ \varepsilon_i(b) = \max\{ k\ge0|\ \tilde e_i b \ne 0 \},
\qquad \varphi_i(b) = \max\{ k\ge0|\ \tilde f_i b \ne 0 \}. $$ In
particular, we denote by $B(\lambda)$ the crystal of the irreducible
highest weight module $V(\lambda)$ with highest weight $\lambda \in
P^+$.

\item Let $(L(\infty), B(\infty))$ be a crystal basis of $U_q^-(\g)$.
Then $B(\infty)$ has a crystal structure, where the maps $\varepsilon_i,\ \varphi_i$ are given by
$$ \varepsilon_i(b) = \max\{ k\ge0|\ \tilde e_i b \ne 0 \}, \qquad \varphi_i(b)=\varepsilon_i(b)+\langle h_i, \wt(b) \rangle .$$

\item For $\lambda \in P$, let us consider $T_\lambda=\{ t_\lambda \}$ with the maps:
\begin{align}
& \wt(t_\lambda)=\lambda,\quad \tilde e_i t_\lambda = \tilde f_i t_\lambda = 0 \text{  for  } i\in I, \nonumber \\
& \varepsilon_i(t_\lambda) = \varphi_i(t_\lambda) = -\infty \text{  for  } i \in I. \nonumber
\end{align}
Then $T_\lambda$ is a crystal.

\item Let $C = \{c \}$. We define the maps
$$ \wt(c)=0,\quad \tilde e_i c= \tilde f_i c = 0, \quad \varepsilon_i(c) =  \varphi_i(c) = 0\ \  (i \in I).$$
Then $C$ is a crystal.
\end{enumerate}
\end{Exa}

\begin{Def}\label{D:mor}
Let $B_1$ and $B_2$ be crystals.

\begin{enumerate}

\item A map $\psi: B_1\rightarrow B_2$ is a {\it crystal morphism}
if it satisfies the following properties:

\begin{enumerate}
\item for $b\in B_1$, we have
\begin{equation*}
\text{$\wt(\psi(b))=\wt(b)$,
$\varepsilon_i(\psi(b))=\varepsilon_i(b)$,
$\varphi_i(\psi(b))=\varphi_i(b)$ for all $i\in I$,}
\end{equation*}
\item for $b\in B_1$ and $i\in I$ with $\tilde f_i b\in B_1$, we have
$\psi(\tilde f_i b)=\tilde f_i \psi(b)$. \label{cond:crysmor2}
\end{enumerate}

\item A crystal morphism $\psi: B_1\rightarrow B_2$ is called {\it
strict} if
\begin{equation*}
\psi(\tilde e_i b)=\tilde e_i\psi(b),\,\, \psi( \tilde f_i b)=\tilde
f_i \psi(b)\quad\text{for all $i\in I$ and $b\in B_1$.}
\end{equation*}
Here, we understand $\psi(0)=0$.

\item $\psi$ is called an {\em embedding} if the underlying map
$\psi: B_1\rightarrow B_2$ is injective.

\end{enumerate}
\end{Def}

Let $B_1$ and $B_2$ be crystals. The {\it tensor
product} $B_1 \otimes B_2$ is defined to be the
set $B_1 \times B_2$ together with the following
maps:
\begin{enumerate}
\item[(a)] $\wt(b_1\otimes b_2) = \wt(b_1)+ \wt(b_2),$
\item[(b)] $\varepsilon_i(b_1\otimes b_2) = \max\{ \varepsilon_i(b_1),\ \varepsilon_i(b_2) - \langle h_i, \wt(b_1) \rangle \}$,
\item[(c)] $\varphi_i(b_1\otimes b_2) = \max\{ \varphi_i(b_2),\ \varphi_i(b_1) + \langle h_i, \wt(b_2) \rangle \}$,
\item[(d)] $\tilde e_i(b_1\otimes b_2) = \left\{
                                     \begin{array}{ll}
                                       \tilde e_i b_1 \otimes b_2 & \hbox{ if } \varphi_i(b_1) \ge \varepsilon_i(b_2), \\
                                       b_1 \otimes \tilde e_i b_2 & \hbox{ if } \varphi_i(b_1) < \varepsilon_i(b_2),
                                     \end{array}
                                   \right.$
\item[(e)] $\tilde f_i(b_1\otimes b_2) = \left\{
                                     \begin{array}{ll}
                                       \tilde f_i b_1 \otimes b_2 & \hbox{ if } \varphi_i(b_1) > \varepsilon_i(b_2), \\
                                       b_1 \otimes \tilde f_i b_2 & \hbox{ if } \varphi_i(b_1) \le \varepsilon_i(b_2).
                                     \end{array}
                                   \right.$
\end{enumerate}
It was shown in \cite{K93} that there is a unique strict crystal embedding
$$B(\lambda) \hookrightarrow B(\infty) \otimes T_{\lambda} \otimes
C$$ sending $u_{\lambda}$ to $\mathbf{1} \otimes t_{\lambda} \otimes
c$. Here, $u_{\lambda}$ is the highest weight element of $B(\lambda)$.  We denote by $\iota_{\lambda}$ the composition of the strict
embedding and the natural projection:
\begin{align} \label{Eqn crystal embedding}
\iota_{\lambda}: B(\lambda) \hookrightarrow B(\infty) \otimes
T_{\lambda} \otimes C \longrightarrow  B(\infty).
\end{align}
Note that $\iota_\lambda$ is injective, but not a
crystal morphism.

\vskip 3em

\section{Path realization}

Let $U'_q(\g)$ be the subalgebra of $U_q(\g)$ generated by $e_i,
f_i, q^{\pm h_i}\ (i\in I)$ and we set $\overline{P}^{\vee} =
\bigoplus_{i=0}^n\Z h_i$, $\overline{\mathfrak{h}} = \C \otimes_\Z
\overline{P}^{\vee}$, $\overline{P} = \bigoplus_{i=0}^n\Z \Lambda_i$
and $\overline{P}^+ = \sum_{i=0}^n\ZZ \Lambda_i$.
Denote by $\cl:P \to \overline{P}$ the natural projection from $P$ to $\overline{P}$.
An abstract crystal $B$ associated with $U'_q(\g)$ is called a
{\it classical crystal}. For $b \in B$, we define
$$ \varepsilon(b) = \sum_{i=0}^n \varepsilon_i(b)\Lambda_i,
\qquad \varphi(b)=\sum_{i=0}^n \varphi_i(b)\Lambda_i. $$

\begin{Def} A {\it perfect crystal of level $\ell$} is a finite
classical crystal $B$ satisfying the following conditions:
\begin{enumerate}
\item[(a)] there exists a finite dimensional $U'_q(\g)$-module with a crystal basis whose crystal graph is isomorphic to $B$,
\item[(b)] $B \otimes B$ is connected,
\item[(c)] there exists a classical weight $\lambda_0 \in \overline{P}$ such that
$$ \wt(B) \subset \lambda_0 + \sum_{i \ne 0} \Z_{\le 0} \alpha_i, \qquad \#(B_{\lambda_0})=1,$$
\item[(d)] for any $b\in B$, we have $\langle c,\ \varepsilon(b) \rangle \ge \ell$,
\item[(e)] for each $\lambda \in \overline{P}_{\ell}^+:= \{ \mu \in \overline{P}^+|\ \langle c,\
\mu \rangle = \ell \}$, there exist unique vectors $b^{\lambda}$
and $b_{\lambda} $ in $B$ such that $\varepsilon(b^\lambda) =
\lambda,\ \varphi(b_\lambda) = \lambda$.
\end{enumerate}
\end{Def}

Given a dominant integral weight $\lambda$ with $\lambda(c)=\ell$
and a perfect crystal $B$ of level $\ell$, it was shown in
\cite{KKMMNN91} that there exists a unique crystal isomorphism,
called the {\it fundamental isomorphism theorem for perfect
crystals},
\begin{align} \label{Eqn fundamental thm of perfect crystals}
\psi: B(\lambda) \overset{\sim} \longrightarrow
B(\varepsilon(b_{\lambda})) \otimes B
\end{align}
sending $u_{\lambda}$ to
$u_{\varepsilon(b_{\lambda})} \otimes b_{\lambda}$. By applying this
crystal isomorphism repeatedly, we get a sequence of crystal
isomorphisms
$$B(\lambda) \overset{\sim} \longrightarrow B(\lambda_1) \otimes
B \overset{\sim} \longrightarrow B(\lambda_2) \otimes B \otimes B
\overset {\sim} \longrightarrow \cdots\cdots,$$ where
$\lambda_0=\lambda$, $b_0=b_{\lambda}$,
$\lambda_{k+1}=\varepsilon(b_k)$, $b_{k+1}=b_{\lambda_{k+1}}$
$(k\ge 0)$. The sequence $\mathbf{p}_{\lambda} :=
(b_k)_{k=0}^\infty$ is called the {\it ground-state path of weight
$\lambda$} and a sequence $\mathbf{p}=(p_k)_{k=0}^\infty$ of
elements $p_k \in B$ is called a {\it $\lambda$-path in $B$} if
$p_k = b_k$ for all $k \gg 0$. We denote by ${\mathcal
P}^{B}(\lambda)$ the set of $\lambda$-paths in $B$, which gives
rise to the {\it path realization} of $B(\lambda)$.

\begin{Thm}[\cite{KKMMNN91}] \label{Thm crystal iso of path realization}
There exists a unique crystal isomorphism $B(\lambda) \overset{\sim}
\longrightarrow \mathcal{P}^{ B}(\lambda)$ which maps
$u_\lambda$ to $\mathbf{p}_\lambda$.
\end{Thm}

We list some examples of perfect crystals of level $1$ and the
corresponding ground-state paths (see \cite{BFKL06, KKMMNN92},
etc).
\begin{Exa} \
\begin{enumerate}
\item The crystal $B^1$ and its ground-state $\mathbf{p}^1_{\Lambda_i}$ of weight $\Lambda_i$ ($i\in I$) are given as
\vskip 1em
\qquad\qquad\qquad\quad
\xymatrix{\mathbf{b}_1 \ar[r]^1 & \mathbf{b}_2 \ar[r]^2 & \cdots \ar[r]^{n-1} & \mathbf{b}_{n} \ar[r]^{n} & \mathbf{b}_{n+1} \ar@/^1.5pc/[llll]_0},
\vskip 0.5em
$$ \mathbf{p}^1_{\Lambda_i} = ( \cdots, \mathbf{b}_1, \mathbf{b}_2, \cdots, \mathbf{b}_n, \mathbf{b}_{n+1}, \mathbf{b}_1, \mathbf{b}_2, \cdots , \mathbf{b}_{i-1}, \mathbf{b}_{i}). $$
\vskip 0.5em

We denote by $\mathcal{P}^1(\Lambda_i)$ the set of all $\Lambda_i$-paths in $B^1$.

\item The crystal $B^n$ and its ground-state $\mathbf{p}^n_{\Lambda_i}$ of weight $\Lambda_i$ ($i\in I$) are given as
\vskip 1em
\qquad\qquad\qquad\quad
\xymatrix{ \overline {\mathbf{b}}_{n+1} \ar[r]^n & \overline{\mathbf{b}}_n \ar[r]^{n-1} & \cdots \ar[r]^{2} & \overline{\mathbf{b}}_{2} \ar[r]^{1} & \overline{\mathbf{b}}_{1} \ar@/^1.5pc/[llll]_0},
\vskip 0.5em
$$ \mathbf{p}^n_{\Lambda_i} = ( \cdots, \overline{\mathbf{b}}_{n+1}, \overline{\mathbf{b}}_n, \cdots, \overline{\mathbf{b}}_2, \overline{\mathbf{b}}_{1}, \overline{\mathbf{b}}_{n+1}, \overline{\mathbf{b}}_n, \cdots , \overline{\mathbf{b}}_{i+2}, \overline{\mathbf{b}}_{i+1}). $$
\vskip 0.5em
We denote by $\mathcal{P}^n(\Lambda_i)$ the set of all $\Lambda_i$-paths in $B^n$.

\item The {\it adjoint crystal} $\adjoint$ is given as follows.

Let
$$\adjoint = \{\emptyset \} \cup \{ { \mathsf{b}_{\pm \alpha_{ij}}} |\ 1 \le i \le j \le n \} \cup \{ h_i|\ i=1,\ldots,n \},$$
where $\alpha_{ij} :=  \alpha_i + \alpha_{i+1} + \cdots+\alpha_j$
for $1 \le i \le j \le n$. We define the $i$-arrows $(i \in I)$ by
\begin{align}
(i \ne 0)\quad & \mathsf{b}_\alpha \buildrel i \over\longrightarrow \mathsf{b}_\beta\ \Leftrightarrow\ \alpha - \alpha_i = \beta, \nonumber \\
&  \mathsf{b}_{\alpha_i} \buildrel i \over\longrightarrow h_i \buildrel i \over\longrightarrow \mathsf{b}_{- \alpha_i}, \nonumber \\
(i=0)\quad & \mathsf{b}_\alpha \buildrel 0 \over\longrightarrow \mathsf{b}_\beta\ \Leftrightarrow\ \alpha + \theta = \beta\ (\alpha, \beta \ne \pm \theta), \nonumber \\
& \mathsf{b}_{- \theta} \buildrel 0 \over\longrightarrow \emptyset \buildrel 0 \over\longrightarrow \mathsf{b}_{\theta}, \nonumber
\end{align}
where $\theta = \alpha_1 + \cdots + \alpha_n$. The crystal $\adjoint$ is a perfect crystal of level $1$ with the ground-state path of weight $\Lambda_0$
$$ \mathbf{p}^{\rm ad}_{\Lambda_0} = (\cdots, \emptyset, \emptyset, \cdots ,\emptyset, \emptyset). $$
There is a crystal isomorphism $\mathfrak{p}^{\rm ad}: B^n \otimes B^1 \to \adjoint $ given by
\begin{align} \label{Eqn isomorphism of perfect crystals}
 \mathfrak{p}^{\rm ad}(  \overline{\mathbf{b}}_j \otimes  \mathbf{b}_i ) =
        \left\{
        \begin{array}{ll}
          \mathsf{b}_{\wt( \overline{\mathbf{b}}_j \otimes  \mathbf{b}_i )} & \hbox{ if }  \wt(\overline{\mathbf{b}}_j \otimes  \mathbf{b}_i)  \ne 0, \\
          h_{i} & \hbox{ if } \wt(\overline{\mathbf{b}}_j \otimes  \mathbf{b}_i) = 0, i \ne n+1, \\
          \emptyset & \hbox{ otherwise }.
        \end{array}
      \right.
\end{align}
We denote by $\mathcal{P}^{\rm ad}(\Lambda_0)$ the set of all $\Lambda_0$-paths in $\adjoint$.

\end{enumerate}
\end{Exa}

\vskip 3em

\section{Combinatorics of Young Walls}

In \cite{HK02, KANG00}, Kang gave a combinatorial realization of
crystal graphs for basic representations of quantum affine
algebras of type $A_n^{(1)}\ (n \ge 1),\ A_{2n-1}^{(2)}\ (n \ge 3),\
D_n^{(1)}\ (n \ge 4),\ A_{2n}^{(2)},\ D_{n+1}^{(2)}\ (n \ge2),$ and
$B_{n}^{(1)}\ (n \ge 3)$ by using new combinatorial objects called {\it Young walls}, which are a generalization of colored Young
diagrams used in \cite{DJKMO89, JMMO91, MM90}. In this work, we
focus on the quantum affine algebra $U_q(A_n^{(1)})$.

The {\it Young wall} $Y^1$ (resp. $Y^n$) is a wall consisting of colored blocks stacked by the following rules:
\begin{enumerate}
\item[(a)] the colored blocks should be stacked in the pattern $\mathsf{P}^1$ (resp. $\mathsf{P}^n$) of weight $\Lambda_k$ given below,
\item[(b)] except for the right-most column, there should be no free space to the right of any block.
\end{enumerate}
The patterns are given as follows.
\vskip 2em
\hskip 3em
\begin{texdraw}
\drawdim em \setunitscale 0.15  \linewd 0.4
\move(0 0)  \lvec(-3 0)\lvec(100 0)         
\move(0 -12) \lvec(-3 -12)\lvec(100 -12)
\move(0 -24) \lvec(-3 -24)\lvec(100 -24)
\move(0 -36) \lvec(-3 -36)\lvec(100 -36)
\move(0 -58) \lvec(-3 -58)\lvec(100 -58)
\move(0 -70) \lvec(-3 -70)\lvec(100 -70)
\move(0 -82) \lvec(-3 -82)\lvec(100 -82)
\move(100 -82) \lvec(100 -82)\lvec(100 3)   
\move(88 -82) \lvec(88 -82)\lvec(88 3)
\move(76 -82) \lvec(76 -82)\lvec(76 3)
\move(54 -82) \lvec(54 -82)\lvec(54 3)
\move(42 -82) \lvec(42 -82)\lvec(42 3)
\move(30 -82) \lvec(30 -82)\lvec(30 3)
\move(18 -82) \lvec(18 -82)\lvec(18 3)
\move(6 -82) \lvec(6 -82)\lvec(6 3)
\htext(93.5 -78){\scriptsize$k$} \htext(74.5 -78){ \scriptsize{${k-1}$}} \htext(59.5 -78){ \scriptsize $\cdots$} \htext(45 -78){ \scriptsize$2$}\htext(34.5 -78){\scriptsize$1$} \htext(22.5 -78){\scriptsize$0$} \htext(10.5 -78.0){\scriptsize$n$}
\htext(88.5 -66){\scriptsize$k+1$} \htext(78.5 -66){ \scriptsize{$k$}} \htext(59.5 -66){ \scriptsize $\cdots$} \htext(45 -66){ \scriptsize$3$}\htext(34.5 -66){\scriptsize$2$} \htext(22.5 -66){\scriptsize$1$} \htext(10.5 -66.0){\scriptsize$0$}
\htext(91.5 -51){ \scriptsize$\vdots$} \htext(79.5 -51){ \scriptsize$\vdots$}
\htext(90 -31){ \scriptsize$n$} \htext(74 -31){ \scriptsize$n-1$}
\htext(90 -19){ \scriptsize$0$} \htext(79 -19){ \scriptsize$n$}
\htext(90 -7){ \scriptsize$1$} \htext(79 -7){ \scriptsize$0$}
\htext(10 -95){ {the pattern $\mathsf{P}^1$ of weight $\Lambda_k$}}

\move(140 0)  \lvec(137 0)\lvec(240 0)         
\move(140 -12) \lvec(137 -12)\lvec(240 -12)
\move(140 -24) \lvec(137 -24)\lvec(240 -24)
\move(140 -36) \lvec(137 -36)\lvec(240 -36)
\move(140 -58) \lvec(137 -58)\lvec(240 -58)
\move(140 -70) \lvec(137 -70)\lvec(240 -70)
\move(140 -82) \lvec(137 -82)\lvec(240 -82)
\move(240 -82) \lvec(240 -82)\lvec(240 3)   
\move(228 -82) \lvec(228 -82)\lvec(228 3)
\move(216 -82) \lvec(216 -82)\lvec(216 3)
\move(194 -82) \lvec(194 -82)\lvec(194 3)
\move(182 -82) \lvec(182 -82)\lvec(182 3)
\move(170 -82) \lvec(170 -82)\lvec(170 3)
\move(158 -82) \lvec(158 -82)\lvec(158 3)
\move(146 -82) \lvec(146 -82)\lvec(146 3)
\htext(233.5 -78){\scriptsize$k$} \htext(214.5 -78){ \scriptsize{${k+1}$}} \htext(199.5 -78){ \scriptsize $\cdots$} \htext(185 -78){ \scriptsize$n$}\htext(174.5 -78){\scriptsize$0$} \htext(162.5 -78){\scriptsize$1$} \htext(150.5 -78.0){\scriptsize$2$}
\htext(228.5 -66){\scriptsize$k-1$} \htext(218.5 -66){ \scriptsize{$k$}} \htext(199.5 -66){ \scriptsize $\cdots$} \htext(180.0 -66){ \scriptsize$n-1$}\htext(174.5 -66){\scriptsize$n$} \htext(162.5 -66){\scriptsize$0$} \htext(150.5 -66.0){\scriptsize$1$}
\htext(231.5 -51){ \scriptsize$\vdots$} \htext(219.5 -51){ \scriptsize$\vdots$}
\htext(230 -31){ \scriptsize$1$} \htext(219 -31){ \scriptsize$2$}
\htext(230 -19){ \scriptsize$0$} \htext(219 -19){ \scriptsize$1$}
\htext(230 -7){ \scriptsize$n$} \htext(219 -7){ \scriptsize$0$}
\htext(150 -95){ {the pattern $\mathsf{P}^n$ of weight $\Lambda_k$}}
\end{texdraw}
\vskip 1em
Note that the heights of the columns of a Young wall $Y$ are weakly decreasing from right to left, so we denote it by
$Y = (y_i)_{i \ge 0}$, where $y_i$ is the $i$-th column of $Y$.
\begin{Def} Let $Y$ be a Young wall corresponding to the pattern $\mathsf{P}^1$ (resp. $\mathsf{P}^n$).
\begin{enumerate}
\item An $i$-block in $Y$ is called a {\it removable $i$-block} if $Y$ remains a Young wall after removing the block.
\item A place in $Y$ is called an {\it admissible $i$-slot} if one may add an $i$-block to obtain another Young wall.
\item A column in $Y$ is said to be {\it $i$-removable} (resp. {\it $i$-admissible}) if the column has a removable $i$-block (resp. an admissible $i$-slot).
\end{enumerate}
\end{Def}
Now we define the action of Kashiwara operators $\tilde e_i, \tilde f_i\ (i\in I)$ on Young walls. Let $Y = (y_k)_{k\ge0}$ be a Young wall corresponding to the pattern $\mathsf{P}^1$ (resp. $\mathsf{P}^n$).
\begin{enumerate}
\item[(a)] To each column $y_k$ of $Y$, we assign
$$
\left\{
  \begin{array}{ll}
    - & \hbox{if $y_k$ is $i$-removable,} \\
    + & \hbox{if $y_k$ is $i$-admissible,}
  \end{array}
\right.
$$
\item[(b)] From this sequence of $+$'s and $-$'s, cancel out all $(+,-)$ pairs to obtain a finite sequence of $-$'s followed by $+$'s. This sequence $(-,\cdots,-,+,\cdots,+)$ is called the {\it $i$-signature} of $Y$.
\item[(c)] We define $\tilde e_i Y$ to be the Young wall obtained from $Y$ by removing the $i$-block corresponding to the rightmost $-$ in the $i$-signature of $Y$. If there is no $-$ in the $i$-signature, we set $\tilde e_i Y = 0 $.
\item[(d)] We define $\tilde f_i Y$ to be the Young wall obtained from $Y$ by adding an $i$-block to the column corresponding to the leftmost $+$ in the $i$-signature of $Y$. If there is no $+$ in the $i$-signature, we set $\tilde f_i Y = 0 $.
\end{enumerate}
We also define
\begin{align*}
\wt(y_j) &= \sum_{i\in I} k_{ij} \alpha_i \ (j \in \Z_{\ge 0}), \\
\wt(Y) &= \Lambda_k - \sum_{j \ge 0} \wt(y_j), \\
\varepsilon_i(Y) &= \text{ the number of $-$'s in the $i$-signature of $Y$}, \\
\varphi_i(Y) &= \text{ the number of $+$'s in the $i$-signature of $Y$},
\end{align*}
where $k_{ij}$ is the number of $i$-blocks in the $j$-th column $y_j$ of $Y$. Note that the height of $y_j$ is $\Ht(\wt(y_j))$.
Let $\mathcal{Y}^1(\Lambda_k)$ (resp.
$\mathcal{Y}^n(\Lambda_k)$) be the set of all Young walls
$Y^1$ (resp. $Y^n$) whose shapes are $(n+1)$-reduced Young
diagrams; i.e., $Y^1 = (y_j)_{j\ge0} \in \mathcal{Y}^1(\Lambda_k)$ (resp. $Y^n = (\overline{y}_j)_{j\ge0} \in \mathcal{Y}^n(\Lambda_k)$) if and only if
$$  \Ht(\wt(y_j)) - \Ht(\wt(y_{j+1}))  < n+1\quad (\text{resp. }  \Ht(\wt(\overline{y}_j)) - \Ht(\wt(\overline{y}_{j+1}))  < n+1 ) $$ for $j \ge 0$. Then $\mathcal{Y}^1(\Lambda_k)$ (resp.
$\mathcal{Y}^n(\Lambda_k)$) has a $U_q(\g)$-crystal structure, and
we have the following theorem.

\begin{Thm} \label{Thm isomorphism of paths and Young walls}
\cite{HK02, KANG00, MM90}
There is a unique crystal isomorphism
$B(\Lambda_k) \overset{\sim} \longrightarrow
\mathcal{Y}^1(\Lambda_k)$ {\rm (}resp.
$\mathcal{Y}^n(\Lambda_k)${\rm )} which maps the highest weight element $u_{\Lambda_k}$ to the
empty Young wall $\emptyset$.
\end{Thm}

Let $\mathbf{p}^1 = (\cdots, \mathbf{b}_{i_2}, \mathbf{b}_{i_1}, \mathbf{b}_{i_0})$ be a $\Lambda_k$-path in $ B^1$.
Consider a Young wall $\mathbf{Y}^1_k(\mathbf{p}^1) = (y_j(\mathbf{p}^1))_{j \ge 0}$ such that the $j$-th column $y_j(\mathbf{p}^1) $ is $\emptyset$ if $j > \Ht(\Lambda_k - \wt(\mathbf{p}^1))$, otherwise $y_j(\mathbf{p}^1)$ is the smallest $j$-th column in $\mathsf{P}^1$ satisfying the following conditions:
\begin{enumerate}
\item[(a)] the top color of $y_j(\mathbf{p}^1)$ is $i_{j}-1$,
\item[(b)] $ y_{j+1}(\mathbf{p}^1) \le y_{j}(\mathbf{p}^1)$.
\end{enumerate}
One can prove that the Young wall $\mathbf{Y}^1_k(\mathbf{p}^1)$ is contained in $ \mathcal{Y}^1(\Lambda_k)$, and the map
\begin{align} \label{Eqn crystal iso from paths P1 to Young walls Y1}
\mathbf{Y}^1_k: \mathcal{P}^1(\Lambda_k) \longrightarrow \mathcal{Y}^1(\Lambda_k)
\end{align}
is a crystal isomorphism. If we set $Y^1 = (y_j)_{j \ge 0} \in \mathcal{Y}^1(\Lambda_k)$, then the inverse image $\mathbf{p}^1$ of $Y^1$ under the crystal isomorphism $\mathbf{Y}^1_k$ is
\begin{align} \label{Eqn iso 1 YWs to paths}
\mathbf{p}^1 = ( \ldots, \mathbf{b}_{a_j}, \ldots, \mathbf{b}_{a_1}, \mathbf{b}_{a_0} ),
\end{align}
where $a_j \equiv \Ht(\wt(y_j)) - j + k \ ( \text{mod}\ n+1)$ for all $j \ge 0$.

In a similar manner, given a $\Lambda_k$-path $\mathbf{p}^n = (\cdots, \overline{\mathbf{b}}_{i_2}, \overline{\mathbf{b}}_{i_1}, \overline{\mathbf{b}}_{i_0})$ in $ B^n$,
we have a Young wall $\mathbf{Y}^n_k(\mathbf{p}^n) = (\overline{y}_j(\mathbf{p}^n))_{j \ge 0}$ such that the $j$-th column $\overline{y}_j(\mathbf{p}^n)$ is $\emptyset$ if $j > \Ht(\Lambda_k - \wt(\mathbf{p}^n))$, otherwise $\overline{y}_j(\mathbf{p}^n)$ is the smallest $j$-th column in $\mathsf{P}^n$ satisfying the following conditions:
\begin{enumerate}
\item[(a)] the top color of $\overline{y}_j(\mathbf{p}^n)$ is $i_{j}$,
\item[(b)] $\overline{y}_{j+1}(\mathbf{p}^n) \le \overline{y}_j(\mathbf{p}^n)$.
\end{enumerate}
One can prove that the Young wall $\mathbf{Y}^n_k(\mathbf{p}^n)$ is contained in $\mathcal{Y}^n(\Lambda_k)$, and the map
\begin{align} \label{Eqn crystal iso from paths Pn to Young walls Yn}
\mathbf{Y}^n_k:\mathcal{P}^n(\Lambda_k) \longrightarrow \mathcal{Y}^n(\Lambda_k)
\end{align}
is a crystal isomorphism. If we set $Y^n = (\overline{y}_j)_{j \ge 0} \in \mathcal{Y}^n(\Lambda_k)$, then the inverse image $\mathbf{p}^n$ of $Y^n$ under the crystal isomorphism $\mathbf{Y}^n_k$ is
\begin{align} \label{Eqn iso n YWs to paths}
\mathbf{p}^n =  (\ldots, \overline{\mathbf{b}}_{b_j}, \ldots, \overline{\mathbf{b}}_{b_1}, \overline{\mathbf{b}}_{b_0}  ),
\end{align}
where $b_j \equiv 1- \Ht(\wt(\overline{y}_j)) + j+k \ ( \text{mod}\ n+1)$ for all $j \ge 0$.


\vskip 2em

\begin{Exa} \label{Exa Young walls} Let $\mathfrak{g}$ be the Kac-Moody algebra of type $A^{(1)}_3$. Fix an element
$$ b = \tilde{f}_0\tilde{f}_2\tilde{f}_1(\tilde{f}_1\tilde{f}_2\tilde{f}_3\tilde{f}_0)^3 u_{\Lambda_0} \in B(\Lambda_0), $$
where $u_{\Lambda_0}$ is the highest weight element of
$B(\Lambda_0)$. Then the $\Lambda_0$-paths $\mathbf{p}^1 \in
\mathcal{P}^1(\Lambda_0)$, $\mathbf{p}^n \in
\mathcal{P}^n(\Lambda_0)$ and $\mathbf{p}^{\rm ad} \in
\mathcal{P}^{\rm ad}(\Lambda_0)$ corresponding to $b$ are given as
\begin{align*}
\mathbf{p}^1 & = ( \ldots, \mathbf{b}_3, \mathbf{b}_1,\mathbf{b}_2, \mathbf{b}_3,
\mathbf{b}_0, \mathbf{b}_1,\mathbf{b}_2, \mathbf{b}_3,\mathbf{b}_0, \mathbf{b}_1,\mathbf{b}_2, \mathbf{b}_3,
\mathbf{b}_0, \mathbf{b}_3), \\
\mathbf{p}^n & = (\ldots, \overline{\mathbf{b}}_3, \overline{\mathbf{b}}_2, \overline{\mathbf{b}}_0, \overline{\mathbf{b}}_1, \overline{\mathbf{b}}_0, \overline{\mathbf{b}}_2, \overline{\mathbf{b}}_1 ), \\
\mathbf{p}^{\rm ad} &= (\ldots, \emptyset, \emptyset, \mathsf{b}_{\alpha_1 + \alpha_2+\alpha_3}, h_1, h_1,
\mathsf{b}_{-\alpha_1 - \alpha_2}),
\end{align*}
which yield the Young walls $Y^1 := \mathbf{Y}^1_0(\mathbf{p}^1)$ and $Y^n := \mathbf{Y}^n_0(\mathbf{p}^n)$ as follows:

{
\centering
$Y^1$ =
\begin{xy} /r1em/:
,(16.8,1.4),{*{2}\xypolygon4{}}
,(16.8,0.0),{*{1}\xypolygon4{}}
,(0,-1.4),{*{0}\xypolygon4{}},(1.4,-1.4),{*{1}\xypolygon4{}},(2.8,-1.4),{*{2}\xypolygon4{}},(4.2,-1.4),{*{3}\xypolygon4{}},(5.6,-1.4),{*{0}\xypolygon4{}}
,(7,-1.4),{*{1}\xypolygon4{}},(8.4,-1.4),{*{2}\xypolygon4{}},(9.8,-1.4),{*{3}\xypolygon4{}},(11.2,-1.4),{*{0}\xypolygon4{}},(12.6,-1.4),{*{1}\xypolygon4{}}
,(14,-1.4),{*{2}\xypolygon4{}},(15.4,-1.4),{*{3}\xypolygon4{}},(16.8,-1.4),{*{0}\xypolygon4{}}
\end{xy}
 , \qquad
$Y^n$ =
\begin{xy} /r1em/:
,(4.2,2.5),{*{2}\xypolygon4{}},(5.6,2.5),{*{1}\xypolygon4{}}
,(1.4,1.1),{*{1}\xypolygon4{}},(2.8,1.1),{*{0}\xypolygon4{}},(4.2,1.1),{*{3}\xypolygon4{}},(5.6,1.1),{*{2}\xypolygon4{}}
,(1.4,-0.3),{*{2}\xypolygon4{}},(2.8,-0.3),{*{1}\xypolygon4{}},(4.2,-0.3),{*{0}\xypolygon4{}},(5.6,-0.3),{*{3}\xypolygon4{}}
,(0,-1.7),{*{0}\xypolygon4{}},(1.4,-1.7),{*{3}\xypolygon4{}},(2.8,-1.7),{*{2}\xypolygon4{}},(4.2,-1.7),{*{1}\xypolygon4{}},(5.6,-1.7),{*{0}\xypolygon4{}}
\end{xy} .
\vskip 1em
}

\end{Exa}

\vskip 3em

\section{Geometric Constructions of Crystal Graphs}

In this section, we review geometric constructions of crystal
bases via quiver varieties. See \cite{ FS03, KS97, L91, N98, St02,
Sv06} for more details.

Let $I = \Z / (n+1)\Z$ and $H$ the set of the arrows such that $i \rightarrow j$ with
$i,j \in I,\ i-j= \pm 1\ $. For $h \in H$, we denote by
$\qin(h)$ (resp. $\qout(h)$) the incoming (resp. outgoing) vertex
of $h$. Define an involution $-:H \to H$ to be the map
interchanging $i\to j$ and $j\to i$. Let $$\Omega = \{ h \in H| \
\qin(h) - \qout(h)  =1   \}$$ so that $H = \Omega \sqcup
\overline{\Omega}$;
i.e., \vskip 2em \hskip 4em
\begin{texdraw}
\drawdim em \setunitscale 0.15  \linewd 0.4
\arrowheadtype t:F

\move(0 0) \htext(-12 -8){$(I,\Omega) = $}
\htext(40 0){$\bullet$} \htext(40.5 4){\scriptsize$0$}
\htext(0 -20){ $\bullet$} \htext(2.5 -24){\scriptsize$1$}
\htext(25 -20){$\bullet$} \htext(25 -24){\scriptsize$2$}
\htext(39.5 -20){$\cdots$}
\htext(55 -20){$\bullet$} \htext(51 -24){\scriptsize$n-1$}
\htext(80 -20){$\bullet$} \htext(81 -24){\scriptsize$n$}
\htext(88 -20){,}

\move(40 0) \arrowheadsize l:2 w:2 \avec(6 -17)
\move(6 -19) \arrowheadsize l:2 w:2 \avec(24 -19)
\move(61 -19) \arrowheadsize l:2 w:2 \avec(78 -19)
\move(80 -17) \arrowheadsize l:2 w:2 \avec(44 0)

\move(120 0) \htext(108 -8){$(I,\overline \Omega) = $}
\htext(160 0){$\bullet$} \htext(160.5 4){\scriptsize$0$}
\htext(120 -20){ $\bullet$} \htext(122.5 -24){\scriptsize$1$}
\htext(145 -20){$\bullet$} \htext(145 -24){\scriptsize$2$}
\htext(159.5 -20){$\cdots$}
\htext(175 -20){$\bullet$} \htext(171 -24){\scriptsize$n-1$}
\htext(200 -20){$\bullet$} \htext(201 -24){\scriptsize$n$}
\htext(208 -20){.}

\move(126 -17) \arrowheadsize l:2 w:2 \avec(160 0)
\move(144 -19) \arrowheadsize l:2 w:2 \avec(126 -19)
\move(198 -19) \arrowheadsize l:2 w:2 \avec(181 -19)
\move(164 0) \arrowheadsize l:2 w:2 \avec(200 -17)

\end{texdraw}
\vskip 1em

Note that our graph is an affine Dynkin graph of type $A_n^{(1)}$. We take the map $\epsilon:H \to \{-1,1\}$ given by
$$ \epsilon(h) = \left\{
                   \begin{array}{ll}
                     1 & \hbox{ if } h \in \Omega, \\
                     -1 & \hbox{ if } h \in \overline{\Omega}.
                   \end{array}
                 \right.
 $$

For $\alpha = \sum_{i=0}^n k_i \alpha_i \in Q^+$, we define the $I$-graded vector space
$$V(\alpha) = \bigoplus_{i=0}^n V_i(\alpha), $$
where $V_i(\alpha)$ is the $\C$-vector space with an ordered basis $v^i(\alpha) = \{ v^i_0,v^i_1 \ldots, v^i_{k_i-1} \}$ for all $i$.
Fix an ordered basis
$$v(\alpha)=\{ v^0_0, \ldots, v^0_{k_0-1}, v^1_{0}, \ldots, v^1_{k_1-1}, \ldots, v^n_0, \ldots  v^n_{k_n-1} \},$$
for $V(\alpha)$ and set
$$  \underline{\dim} V(\alpha) = \sum_{i=0}^n k_i \alpha_i = \alpha. $$
In a similar manner, for $\lambda = \sum_{i=0}^n w_i \Lambda_i \in P^+ $, we define the $I$-graded vector space
$$ W(\lambda) = \bigoplus_{i=0}^n W_i(\lambda),$$
where $W_i(\lambda)$ is a $\C$-vector space of dimension $w_i$.

Given $\alpha \in Q^+$, we set $V = V(\alpha)$ (resp. $V_i =
V_i(\alpha)\ (i=I)$) and let
$$ E(\alpha) = E_\Omega (\alpha ) \oplus E_{\overline{\Omega}}(\alpha),$$
where
\begin{align}
E_{\Omega}(\alpha ) &= \bigoplus_{h\in \Omega}\Hom(V_{\qout(h)}, V_{\qin(h)}) = \bigoplus_{i \in I}\Hom(V_{i-1},\ V_{i}), \nonumber \\
E_{\overline{\Omega}}(\alpha) &= \bigoplus_{h\in \overline{\Omega}}\Hom(V_{\qout(h)}, V_{\qin(h)}) = \bigoplus_{i \in I}\Hom(V_{i},\ V_{i-1}).   \nonumber
\end{align}
Let us denote by $\pi_{\Omega}$ (resp. $\pi_{\overline{\Omega}}$) the natural projection from $E(\alpha)$ to $E_{\Omega}(\alpha)$ (resp. $E_{\overline{\Omega}}(\alpha)$). For $\chi \in E(\alpha)$, if there is no danger of confusion, we write $x=(x_i\in \Hom(V_{i-1},\ V_{i}))_{i\in I}$ (resp. $\overline x=( \overline{x}_i \in \Hom(V_{i},\ V_{i-1}))_{i\in I}$) for $\pi_{\Omega}(\chi)$ (resp. $\pi_{\overline \Omega}(\chi)$). The matrix representation of $x\in E_{\Omega}(\alpha)$ in the ordered basis $v(\alpha)$ is given as
$$ x = \left(
         \begin{array}{ccccc}
           0   &  &   \cdots  &       0  & x_0 \\
           x_1  & 0   &  &         &  0 \\
           0   &   x_2  & 0   &   & \vdots \\
           \vdots    &     &  \ddots   & \ddots  & 0 \\
           0 &  \cdots   &  0   &     x_n    & 0 \\
         \end{array}
       \right),
 $$
where $x_i$ is the matrix representation of $x|_{V_{i-1}}: V_{i-1} \to V_i$ in the ordered bases $v^{i-1}(\alpha)$ and $v^{i}(\alpha)$. We also may consider the matrix representation of $\overline{x}\in E_{\overline{\Omega}}(\alpha ) $ in the same manner:
$$ \overline{x} = \left(
         \begin{array}{ccccc}
           0   & \overline{x}_1 & 0    &     \cdots  & 0 \\
            & 0   & \overline{x}_2 &         &   \vdots \\
            \vdots   &     & 0   &  \ddots & 0 \\
             0  &     &    & \ddots  & \overline{x}_n  \\
           \overline{x}_0 &  0   &  \cdots   &     0    & 0 \\
         \end{array}
       \right),
 $$
where $\overline{x}_i$ is the matrix representation of $\overline{x}|_{V_i}: V_i \to V_{i-1}$ in the ordered bases $v^i(\alpha)$ and $v^{i-1}(\alpha)$.

For $s \in I,\ 0\le i < k_{s-1},\ 0\le j < k_{s}$, define the linear map $\mathcal{E}^s_{ij}:V_{s-1} \to V_{s}$ by
$$
 \mathcal{E}^{s}_{ij}(v^{s-1}_k) = \left\{
                       \begin{array}{ll}
                         v^{s}_j & \hbox{ if } k=i, \\
                         0 & \hbox{ otherwise.}
                       \end{array}
                   \right.
$$ Similarly, for $s \in I,\ 0\le i < k_{s},\ 0\le j < k_{s-1}$, define the linear map $\overline{\mathcal{E}}^s_{ij}:V_{s} \to V_{s-1}$ by
$$
 \overline{\mathcal{E}}^{s}_{ij}(v^{s}_k) = \left\{
                       \begin{array}{ll}
                         v^{s-1}_j & \hbox{ if } k=i, \\
                         0 & \hbox{ otherwise.}
                       \end{array}
                   \right.
$$

The algebraic group $G(\alpha):= \prod_{i \in I} \Aut(V_i) \subset \Aut(V)$ acts on $E(\alpha)$ by
$(g,\chi)  = g \chi g^{-1}$ for $g \in G(\alpha),\ \chi \in E(\alpha)$.
Let $\langle \cdot,\cdot \rangle$ be the nondegenerate, $G(\alpha)$-invariant, sympletic form on $E(\alpha)$ defined by $$ \langle \chi,\chi' \rangle = \sum_{h\in H} \epsilon(h) \tr(\chi_h {\chi'}_{\overline h})$$
for $\chi,\chi' \in E(\alpha)$. Note that $E(\alpha)$ may be viewed as the cotangent bundle of $E_{\Omega}(\alpha)$ (resp. $E_{\overline{\Omega}}(\alpha)$) under this form. The {\it moment map} $\mu=(\mu_i: E(\alpha) \to \End(V_i))_{i\in I}$ is given by
$$ \mu_i(\chi) = \sum_{h\in H,\ \qin(h)=i} \epsilon(h)\chi_h \chi_{\overline h}\
=\ x_{i} \overline{x}_{i} - \ \overline{x}_{i+1} x_{i+1}   $$
for $\chi = x + \overline x \in E(\alpha)$. Note that
\begin{align}\label{Eqn commuting property of chi}
\mu_i(\chi)=0\ \text{for all } i \in I\quad \text{ if and only if } \quad [x, \overline x] = x \overline x - \overline x  x = 0.
\end{align}
Here, $x \overline{x} = (x_{i} \overline{x}_{i}: V_i \rightarrow
V_i)_{i\in I}$ and $\overline{x} x = (\overline{x}_{i+1} x_{i+1} :
V_i \rightarrow V_i)_{i\in I}$.

An element $\chi \in E(\alpha)$ is {\it nilpotent} if there exists
an $N \ge 2$ such that for any sequence $h_1, \ldots, h_N \in H$
satisfying $\qin(h_i)=\qout(h_{i+1})\ (i=1,\ldots,N-1)$, the
composition map $\chi_{h_N} \cdots \chi_{h_1}$ is zero. We define
{\it Lusztig's quiver variety} to be
$$\Lambda(\alpha) = \{\chi \in E(\alpha)|\ \chi: \text{nilpotent},\  \mu_i(\chi)=0 \text{ for all }i\in I \}.$$
We denote by $\Irr \Lambda(\alpha)$ the set of all irreducible components of $\Lambda(\alpha)$

For a pair $(k' \le k)$ of integers, let $V(k', k) = \bigoplus_{i\in I} V_i(k', k)$ be the $I$-graded vector space with basis $\{ \mathsf{e}_j|\ k' \le j \le k \}$ such that $V_i(k', k) = \Span_\C\{ \mathsf{e}_j|\ j \equiv i \ {\rm (mod }\ n+1 {\rm )}  \}$ for $i\in I$. Consider the $\C$-linear map $x(k',k):V(k', k) \to V(k', k)$ sending $\mathsf{e}_i$ to $\mathsf{e}_{i+1}$, where $\mathsf{e}_{k+1}=0$. Then it is clear that the representation $(V(k', k), x(k', k))$ of the quiver $(I,\Omega)$ is indecomposable and nilpotent. Note that the isomorphism class of $(V(k', k), x(k', k))$ does not change when $k'$ and $ k$ are simultaneously translated by a multiple $n+1$. Moreover, any indecomposable nilpotent finite-dimensional representation of the quiver $(I,\Omega)$ is isomorphic to $(V(k', k), x(k', k))$ for some pair $(k' \le k) $. Let $\mathbf{Z}$ be the set of all pairs $(k' \le k)$ of integers defined up to simultaneous translation by a multiple of $n+1$ and let $\tilde{ \mathbf{Z}}$ be the set of all functions from $\mathbf{Z}$ to $\Z_{\ge 0}$ with finite support. Note that $\tilde{ \mathbf{Z}}$ naturally corresponds to isomorphism classes of nilpotent finite-dimensional representations of the quiver $(I,\Omega)$. The set of $G(\alpha)$-orbits on the set of nilpotent elements in $E_{\Omega}(\alpha)$ is naturally indexed by the subset $\tilde{\mathbf{Z}}(\alpha)$ of $\tilde{\mathbf{Z}}$ such that, for $f \in \tilde{\mathbf{Z}}(\alpha)$,
$$ \sum_{k' \le k} f(k', k) \cdot \# \{ j|\ k' \le j \le k,\ j\equiv i\ ( \text{mod}\ n+1) \} = \dim V_i\ (i\in I).$$
Here the sum is taken over all $k' \le k$ up to simultaneous translation by a multiple of $n+1$.
An element $f \in \tilde{\mathbf{Z}}(\alpha)$ is {\it aperiodic} if, for any $k' \le  k$, not all integers
$f(k',k),f(k'+1,k+1), \ldots, f(k'+n, k+n) $ are greater than zero. For any $f \in  \tilde{\mathbf{Z}}(\alpha)$, let $\mathcal{C}_f$ be the conormal bundle of the $G(\alpha)$-orbit corresponding to $f$, and let $\overline{\mathcal{C}}_f$ be the closure of $\mathcal{C}_f$. Then we have

\begin{Thm}[\cite{L91}] \label{Thm Lusztig conormal bundle}
The map $f \mapsto \overline{\mathcal{C}}_f$ is a 1-1 correspondence between the set of aperiodic elements in $\tilde{\mathbf{Z}}(\alpha)$ and $ \Irr \Lambda(\alpha)$.
\end{Thm}

In a similar manner, for a pair $(k \ge k')$ of integers, let $\overline{x}(k,k'):V(k', k) \to V(k', k)$ be the $\C$-linear map sending $\mathsf{e}_i$ to $\mathsf{e}_{i-1}$, where $\mathsf{e}_{k'-1}=0$. Then the representation $(V(k', k), \overline{x}(k, k'))$ of the quiver $(I,\overline{\Omega})$ is indecomposable and nilpotent, and the isomorphism class of $(V(k', k), \overline{x}(k, k'))$ does not change when $k$ and $k'$ are simultaneously translated by a multiple $n+1$. Any indecomposable nilpotent finite-dimensional representation of the quiver $(I,\overline{\Omega})$ is isomorphic to $(V(k', k), \overline{x}(k, k'))$ for some pair $(k \ge k') $.
Let $\overline{\mathbf{Z}}$ be the set of all pairs $(k \ge k')$ of integers defined up to simultaneous translation by a multiple of $n+1$ and let $\tilde{\overline{\mathbf{Z}}}$ be the set of all functions from $\overline{\mathbf{Z}}$ to $\Z_{\ge 0}$ with finite support.
Then the set of $G(\alpha)$-orbits on the set of nilpotent elements in $E_{\overline{\Omega}}(\alpha)$ is naturally indexed by the subset $\tilde{\overline{\mathbf{Z}}}(\alpha)$ of $\tilde{\overline{\mathbf{Z}}}$ such that, for $f \in \tilde{\overline{\mathbf{Z}}}(\alpha)$,
$$ \sum_{k \ge k'} f(k, k') \cdot \# \{ j|\ k \ge j \ge k',\ j\equiv i\ ( \text{mod}\ n+1) \} = \dim V_i\ (i\in I).$$
Here the sum is taken over all $k \ge k'$ up to simultaneous
translation by a multiple of $n+1$. An element $f \in
\tilde{\overline{\mathbf{Z}}}(\alpha)$ is {\it aperiodic} if, for
any $k \ge  k'$, not all integers $f(k,k'),f(k+1,k'+1), \ldots,
f(k+n, k'+n) $ are greater than zero. Then one can show that there
is  a 1-1 correspondence between the set of aperiodic elements in
$\tilde{\overline{\mathbf{Z}}}(\alpha)$ and $\Irr\Lambda(\alpha)$.

Moreover, Kashiwara and Saito \cite{KS97} gave a crystal structure on
$$ \mathbb{B}(\infty) := \bigsqcup _{\alpha \in Q^+} \Irr \Lambda(\alpha),$$
and proved the following theorem.

\begin{Thm}[\cite{KS97}]\label{Thm K-S geometric crystal} There is a unique crystal isomorphism
$  \mathbb{B}(\infty) \cong B(\infty).$
\end{Thm}

Now we introduce a description of Nakajima's quiver varieties presented in \cite{N94}. Given $\alpha \in Q^+$ and $
\lambda \in P^+$, we set $W = W(\lambda)$ (resp. $W_i = W_i(\lambda)$) and let
$$ E(\lambda, \alpha) = \Lambda(\alpha) \times \sum_{i \in I} \Hom(V_i, W_i). $$
The group $G(\alpha)$ acts on $E(\lambda, \alpha)$ by $(g, (\chi,
\mathfrak{t})) = (g \chi g^{-1}, \mathfrak{t} g^{-1})$. For $\chi
\in \Lambda(\alpha)$, an $I$-graded subspace $S$ of $V(\alpha)$ is
{\it $\chi$-stable} if $ \chi_{h}(S_{\qout(h)}) \subset
S_{\qin(h)} $ for all $h \in H $. An element $(\chi, \mathfrak{t})
\in E(\lambda, \alpha )$ is called a {\it stable point} of
$E(\lambda, \alpha)$ if it satisfies the following conditions: if
$S$ is a $\chi$-stable subspace of $V$ with $\mathfrak{t}_i(S_i)=0\
(i\in I)$, then $S = 0$. Let $E(\lambda, \alpha)^{st}$ be the set
of all stable points of $E(\lambda, \alpha)$, and define
$$\Lambda(\lambda, \alpha) = E(\lambda, \alpha)^{st} / G(\alpha).$$

Let $\Irr \Lambda(\lambda, \alpha)$ (resp. $\Irr E(\lambda, \alpha)$) be the set of all irreducible components of $\Lambda(\lambda, \alpha)$ (resp. $E(\lambda, \alpha)$). Since $\Irr \Lambda(\lambda, \alpha)$ can be identified with
$$ \{ Z \in \Irr E(\lambda, \alpha)|\ Z \cap E(\lambda, \alpha)^{st} \ne \emptyset \}, $$
each irreducible component $X$ in $\Irr \Lambda(\lambda, \alpha)$
can be written as
$$ X = \left( \left( X_0 \times \sum_{i\in I} \Hom(V_i, W_i)  \right) \cap E(\lambda, \alpha)^{st}
\right) / G(\alpha) $$
for some irreducible component $X_0$ in $\Irr \Lambda(\alpha)$.

In \cite{St02}, Saito gave a crystal structure on
$$ \mathbb{B}(\lambda) := \bigsqcup_{\alpha \in Q^+} \Irr \Lambda(\lambda, \alpha) ,$$
and proved the following theorem.

\begin{Thm}[\cite{St02}] \label{Thm saito geometric crystal}
There is a unique crystal isomorphism $\mathbb{B}(\lambda) \cong B(\lambda)$.
\end{Thm}

In \cite{FS03}, Frenkel and Savage gave an enumeration of $\Irr
\Lambda(\lambda, \alpha)$ in terms of Young and Maya diagrams for
type $A_n^{(1)}$.
Combining Theorem \ref{Thm K-S geometric crystal} and Theorem
\ref{Thm saito geometric crystal} with $\eqref{Eqn crystal
embedding}$, we obtain an injective map
\begin{align} \label{Eqn crystal embedding in geom}
 \iota_{\lambda}: \mathbb{B}(\lambda) \hookrightarrow \mathbb{B}(\infty).
\end{align}
For each irreducible component $X_0 \in \iota_{\Lambda_k}(
\mathbb{B}(\Lambda_k))$, Frenkel and Savage constructed a special
point in $X_0 \times \sum_{i \in I} \Hom(V_i, W_i)$ which is not
killed by the stability condition, and showed that there is a 1-1
correspondence between the set of such special points and the set of
$(n+1)$-reduced colored Young diagrams. Savage later established a
crystal isomorphism between $\mathbb{B}(\lambda)$ and Young walls
for quantum affine algebras of  type $A_n^{(1)}$ and $D_n^{(1)}$ in
\cite{Sv06}.

We briefly recall the result of \cite{FS03} for type $A_n^{(1)}$ in
terms of Young walls. Note that the orientation appeared in
\cite{FS03} is $\overline{\Omega}$. Take a Young wall $Y^n \in
\mathcal{Y}^n(\Lambda_k)$ such that $\wt(Y^n) = \alpha$. Let $l_i$
be the length of the $i$-th row of the Young wall $Y^n$ $(i\ge1)$
and let $N$ be the height of $Y^n$. Set
$$ A_{Y^n} := \{ (l_i - i +k, 1-i+k)\ |\ 1 \le i \le N  \} \subset \overline{\mathbf{Z}}$$
and consider the function $f \in \tilde{\overline{\mathbf{Z}}}(\alpha)$ given by
$$ f(s,s') = \left\{
               \begin{array}{ll}
                 1 & \hbox{ if } (s,s') \in A_{Y^n}, \\
                 0 & \hbox{otherwise.}
               \end{array}
             \right.
 $$
Note that $f$ is aperiodic. Let $\overline{\mathcal{C}}_f$ be the
closure of the conormal bundle of the $G(\alpha)$-orbit
$\mathcal{O}_f$ in $E_{\overline{\Omega}}(\alpha)$ corresponding to
$f$, and define the irreducible component
$$ X_{Y^n} := \left(  \left( \overline{\mathcal{C}}_f \times \sum_{i\in I} \Hom(V_i, W_i)  \right) \cap E(\Lambda_k, \alpha)^{st}  \right) / G(\alpha) \in \Irr\Lambda(\Lambda_k, \alpha) . $$
By \cite[Theorem 5.5]{FS03}, the map $Y^n \mapsto X_{Y^n}$ is a 1-1
correspondence between
$$\{ Y^n \in \mathcal{Y}^n(\Lambda_k) |\ \wt(Y^n) = \alpha \}\ \text{ and } \ \Irr \Lambda(\Lambda_k, \alpha).$$
Moreover, it is proved in \cite[Theorem 8.4.]{Sv06} that the map
$Y^n \mapsto X_{Y^n}$ from $\mathcal{Y}^n(\Lambda_k)$ to
$\mathbb{B}(\Lambda_k)$ is a crystal isomorphism. We would like to
point out that $ \overline{\mathcal{C}}_f =
\iota_{\Lambda_k}(X_{Y^n}) $.

Now we construct an element in the $G(\alpha)$-orbit $\mathcal{O}_f$ in $E_{\overline{\Omega}}(\alpha)$ from the Young wall $Y^n$.
Let  $\overline{\mathfrak{b}}_{ij}$ be
the $i$-th block from bottom in the $j$-th column of $Y^n$. Let $\Color(\overline{\mathfrak{b}}_{ij})$ be the color of
$\overline{\mathfrak{b}}_{ij}$, which is an element in $I$. Define
\begin{align}
o(\overline{\mathfrak{b}}_{ij}) &:= \# \{ \overline{\mathfrak{b}}_{rs} \in Y^n|\   \Color(\overline{\mathfrak{b}}_{rs}) = \Color(\overline{\mathfrak{b}}_{ij}),\  (r,s) \prec (i,j)    \},  \nonumber
\end{align}
where $\prec$ is the lexicographical order; i.e., $ (r,s) \prec (i,j) $ if and only if $r<i$ or ($r=i$ and $s<j$). We define
\begin{align} \label{Eqn Young Walls Bn}
 \overline{x}(Y^n) := \sum_{\overline{\mathfrak{b}}_{ij}\in Y^n,\ j > 0} \overline{\mathcal{E}}^{\Color(\overline{\mathfrak{b}}_{i,j})}_{o(\overline{\mathfrak{b}}_{i,j}),\ o(\overline{\mathfrak{b}}_{i,j-1})} \in E_{\overline{\Omega}}(\alpha).
\end{align}
For $1 \le i \le N$, we denote by $J_i$ the subspace of $V(\alpha)$ generated by
$$\{ v^{\Color(\overline{\mathfrak{b}}_{ij})}_{o(\overline{\mathfrak{b}}_{ij})} |\ 0 \le j < \l_i  \} .$$
By construction, one can show that $J_i$ is invariant under $\overline{x}(Y^n)$ and
the representation $ (J_i,\ \overline{x}(Y^n) |_{J_i}) $ of the quiver $(I,\overline{\Omega})$ is isomorphic to the representation $( V(1-i+k, l_i-i+k),\ \overline{x}(l_i-i+k, 1-i+k) )$ of the quiver $(I,\overline{\Omega})$ for $1 \le i < N$. Here, $\overline{x}(Y^n) |_{J_i}$ is the restriction of $\overline{x}(Y^n)$ on the invariant subspace $J_i$.
Hence $\overline{x}(Y^n)$ is contained in the $G(\alpha)$-orbit $\mathcal{O}_f$ corresponding $f$, which yields
\begin{align} \label{Eqn Conormal bundle of G-orbit of Yn}
\iota_{\Lambda_k}(X_{Y^n}) = \text{ the closure of the conormal bundle of the $G(\alpha)$-orbit of $\overline{x}(Y^n)$}.
\end{align}
By a direct computation, for $t \in \Z_{\ge 0 }$, we have
\begin{align} \label{Eqn ker of Yn}
\ker(\overline{x}(Y^n) )^t & = \bigoplus_{i=1}^{N} \ker (\overline{x}(Y^n)|_{J_i})^t \nonumber \\
& = \bigoplus_{i=1}^{N} \Span_{\C}\{ v^{\Color(\overline{\mathfrak{b}}_{ij})}_{o(\overline{\mathfrak{b}}_{ij})} |\
\overline{\mathfrak{b}}_{ij} \in \text{the $i$-th row of }Y^n,\ j < t \} \nonumber \\
&= \Span_{\C}\{ v^{\Color(\overline{\mathfrak{b}}_{ij})}_{o(\overline{\mathfrak{b}}_{ij})} |\
\overline{\mathfrak{b}}_{ij} \in Y^n,\ j < t \}.
\end{align}

In the same manner, we take a Young wall $Y^1 \in \mathcal{Y}^1(\Lambda_k)$ such that $\wt(Y^1)=\alpha$.
Denote by $X_{Y^1}$ the image of $Y^1$ under the crystal isomorphism
$$\mathcal{Y}^1(\Lambda_k)\overset{\sim}{\longrightarrow} \mathbb{B}(\Lambda_k) .$$
Let  $\mathfrak{b}_{ij}$ be the $i$-th block from bottom in the $j$-th column of $Y^1$, and
$\Color(\mathfrak{b}_{ij})$
the color of $\mathfrak{b}_{ij}$. Set
\begin{align}
o(\mathfrak{b}_{ij}) &:= \# \{ \mathfrak{b}_{rs} \in Y^1|\   \Color(\mathfrak{b}_{rs}) = \Color(\mathfrak{b}_{ij}),\  (r,s) \prec (i,j)    \},  \nonumber
\end{align}
where $\prec$ is the lexicographical order, and define
\begin{align} \label{Eqn Young Walls B1}
 x(Y^1) := \sum_{\mathfrak{b}_{ij}\in Y^1,\ j > 0} \mathcal{E}^{\Color(\mathfrak{b}_{i,j-1})}_{o(\mathfrak{b}_{i,j}),\ o(\mathfrak{b}_{i,j-1})} \in E_{\Omega}(\alpha).
\end{align}
Then we have
\begin{align} \label{Eqn Conormal bundle of G-orbit of Y1}
\iota_{\Lambda_k}(X_{Y^1})  =  \text{ the closure of the conormal bundle of the $G(\alpha)$-orbit of $x(Y^1)$}.
\end{align}
Moreover, we obtain
\begin{align}  \label{Eqn ker of Y1}
& \ker(x(Y^1) )^t = \Span_{\C}\{ v^{\Color(\mathfrak{b}_{ij})}_{o(\mathfrak{b}_{ij})} |\
\mathfrak{b}_{ij} \in Y^1,\ j < t \}
\end{align}
for $t\in \Z_{\ge 0}$.


\vskip 2em

\begin{Exa} \label{Exa the points from Young walls} We use the same notations as in Example \ref{Exa Young walls}. Set
$$\alpha := \Lambda_0-\wt(b)= 4\alpha_0 + 4\alpha_1 + 4\alpha_2 +  3\alpha_3,$$
and let $X$ be the irreducible component in $\mathbb{B}(\Lambda_0)$ corresponding to $b$ via the crystal isomorphism given in Theorem \ref{Thm saito geometric crystal}. Then we have
\begin{align*}
x(Y^1) &= (\mathcal{E}^{0}_{00}+ \mathcal{E}^{0}_{11} + \mathcal{E}^{0}_{22}) + (\mathcal{E}^{1}_{10}+ \mathcal{E}^{1}_{21} + \mathcal{E}^{1}_{32}) + (\mathcal{E}^{2}_{00}+ \mathcal{E}^{2}_{11} + \mathcal{E}^{2}_{22}) + (\mathcal{E}^{3}_{00}+ \mathcal{E}^{3}_{11} + \mathcal{E}^{3}_{22}), \\
\overline{x}(Y^n) &= (\overline{\mathcal{E}}^{0}_{10} + \overline{\mathcal{E}}^{0}_{21} + \overline{\mathcal{E}}^{0}_{32}) + (\overline{\mathcal{E}}^{1}_{00} + \overline{\mathcal{E}}^{1}_{12} + \overline{\mathcal{E}}^{1}_{23}) + (\overline{\mathcal{E}}^{2}_{00} + \overline{\mathcal{E}}^{2}_{11} + \overline{\mathcal{E}}^{2}_{33}) + (\overline{\mathcal{E}}^{3}_{00} + \overline{\mathcal{E}}^{3}_{22} ),
\end{align*}
and $\iota_{\Lambda_0}(X)$ is the closure of the conormal bundle of
the $G(\alpha)$-orbit of $x(Y^1)$ (resp. $\overline{x}(Y^n)$).
However, we note that $$x(Y^1) + \overline{x}(Y^n) \notin
E(\alpha)$$ since $[x(Y^1), \overline{x}(Y^n)] \ne 0$.

\end{Exa}

\vskip 3em

\section{Quiver Varieties and the Perfect Crystals $B^1$, $B^n$}

In this section, we give an explanation of the 1-1 correspondence
between the geometric realization $\mathbb{B}(\Lambda_k)$ and the
path realization of the crystal $B(\Lambda_k)$ associated with the
perfect crystals $B^1$ and $ B^n$, and give a geometric
interpretation of the fundamental theorem of perfect crystals in the
case of the perfect crystals $B^1$ and $B^n$. Let $\alpha \in Q^+$
and let $\lambda$ be a dominant integral weight of level 1. Choose
an irreducible component $X$ in $\Irr \Lambda(\lambda, \alpha)$. For
a generic point $\chi = x + \overline{x} \in \iota_{\lambda}(X)$, we
will give an explicit description of the $\lambda$-path in $B^1$
(resp. $B^n$) corresponding to $X$ using the dimensions of the
spaces $\ker x^{i+1} / \ker x^i$ (resp. $\ker \overline{x}^{i+1} /
\ker \overline{x}^i $) for $i \ge 0$. For this purpose, we need a
couple of lemmas.

\begin{Lem} \label{Lem kernel is invariant}
Let $X_0$ be an irreducible component in $ \Irr \Lambda(\alpha)$. Then,
for any $\chi = x + \overline{x} \in X_0$ and $k \in \mathbb{Z}_{\ge 0}$, we have
\begin{enumerate}
\item[(a)] $\ker(x \overline{x})^k = \ker(\overline{x} x) ^k $,
\item[(b)] $\ker x^k$ and $\ker \overline{x}^k$ are $\chi$-stable,
\item[(c)] $\ker(x \overline{x})^k$ is $\chi$-stable.
\end{enumerate}
\end{Lem}
\begin{proof}
Let $\chi = x + \overline{x} \in X_0$ and $k \in \mathbb{Z}_{\ge 0} $. By $\eqref{Eqn commuting property of chi}$ we have $[x, \overline{x}] = 0,$
which yields
$$   \chi (\ker x^k)  \subset \ker x^k, \  \chi (\ker \overline{x}^k)  \subset \ker \overline{x}^k\ \text{ and }\ \chi (\ker(x \overline{x})^k)  \subset \ker(x \overline{x})^k .$$
Our assertion follows from the fact that $\ker(x \overline{x})^k,\
\ker x^k$ and $\ker \overline{x}^k$ are $I$-graded vector spaces.
\end{proof}

\begin{Lem} \label{Lem kernels of x and overline x}
For each $X_0 \in  \Irr \Lambda(\alpha)$, there exists an open subset $U \subset X_0$ such that
\begin{align} \label{Eqn constant kernel - generic properties of X}
\ker x^k \cong \ker {x'}^k\ \text{ and }\ \ker \overline{x}^k \cong \ker {\overline{x}'}^k
\end{align}
for any $\chi = x+\overline{x},\ \chi' = x'+\overline{x}' \in U$ and $k \in \Z_{\ge 0}$.
\end{Lem}
\begin{proof}
By Theorem \ref{Thm Lusztig conormal bundle}, there is an open subset $U_1 \subset X_0$ such that $\pi_{\Omega}(U_1)$ is contained in the $G(\alpha)$-orbit of some element in $E_{\Omega}(\alpha)$. In the same manner, there is an open subset $U_2  \subset X_0$ such that $\pi_{\overline{\Omega}}(U_2)$ is contained in the $G(\alpha)$-orbit of some element in $ E_{\overline{\Omega}}(\alpha)$. Set $U = U_1 \cap U_2 \subset X_0$. Then, by construction, for any $\chi = x + \overline{x},\ \chi' = x' + \overline{x}' \in U$, there exist $g, \overline{g} \in G(\alpha)$ such that
$$ x = gx'g^{-1} \quad \text{ and } \quad  \overline{x} = \overline{g}\overline{x}'\overline{g}^{-1} ,$$
which yield, for any $k \in \Z_{\ge 0}$,
$$ \ker x^k = \ker (g{x'}g^{-1})^{k} = g (\ker {x'}^{k})\ \ \text{ and } \ \ \ker \overline{x}^k = \ker (\overline{g}{\overline{x}'}\overline{g}^{-1})^{k} = \overline{g} (\ker {\overline{x}'}^{k}). $$
\end{proof}

An element $\chi \in X_0$ in the open subset $U \subset X_0$ in
Lemma \ref{Lem kernels of x and overline x} will be called a {\it
generic point}. Thanks to Lemma \ref{Lem kernels of x and overline
x}, we may consider
$$ \underline{\dim}(\ker x^k)\ \text{ and }\ \underline{\dim}(\ker x^{k+1} /\ker x^k )\
(\text{resp. } \underline{\dim}(\ker \overline{x}^k)\ \text{ and }\
\underline{\dim}(\ker \overline{x}^{k+1} /\ker \overline{x}^k ) )$$
for a generic point $\chi = x + \overline{x} $ in an irreducible
component $X_0 \in \Irr \Lambda(\alpha)$.
Recall the injective map given in $\eqref{Eqn crystal embedding in
geom}$
\begin{align*}
 \iota_{\Lambda_k}: \mathbb{B}(\Lambda_k) \hookrightarrow \mathbb{B}(\infty)
\end{align*}
for $0 \le k \le n$. Applying Lemma \ref{Lem kernel is invariant}
and Lemma \ref{Lem kernels of x and overline x} to $\eqref{Eqn
Young Walls B1}$ and $\eqref{Eqn Young Walls Bn}$, we obtain the
following theorem.
\begin{Thm} \label{Thm paths of B1 and Bn} {\rm(}cf. \cite{FS03}{\rm)}
Let
$$ \mathbf{p}_k^1: \mathbb{B}(\Lambda_k) \longrightarrow \mathcal{P}^1(\Lambda_k)
\ \text{ {\rm (}resp. } \mathbf{p}_k^n: \mathbb{B}(\Lambda_k) \longrightarrow \mathcal{P}^n(\Lambda_k) \ \text{\rm )}$$
be the unique crystal isomorphism given by Theorem \ref{Thm crystal
iso of path realization} and Theorem \ref{Thm saito geometric
crystal}, and take an irreducible component $X \in
\mathbb{B}(\Lambda_k)$. Then, for a generic point $\chi = x +
\overline{x} \in \iota_{\Lambda_k}(X)$, we have
\begin{enumerate}
\item[(a)]
$$ \mathbf{p}^1_k(X) = (\ldots, \mathbf{b}_{a_i}, \ldots, \mathbf{b}_{a_1}, \mathbf{b}_{a_0}  ), $$
where $a_i \equiv \dim(\ker x^{i+1} /\ker x^i ) - i + k \ ( \text{mod}\ n+1)$ for all $i \ge 0$,
\item[(b)]
$$ \mathbf{p}^n_k(X) = (\ldots, \overline{\mathbf{b}}_{b_i}, \ldots, \overline{\mathbf{b}}_{b_1}, \overline{\mathbf{b}}_{b_0}  ), $$
where $b_i \equiv 1-\dim( \ker \overline{x}^{i+1} /\ker \overline{x}^i ) + i+k \ ( \text{mod}\ n+1)$ for all $i \ge 0$.
\end{enumerate}
\end{Thm}

\begin{proof}
Let $U$ be an open subset of $\iota_{\Lambda_k}(X)$ as in Lemma \ref{Lem kernels of x and overline x}. By Lemma \ref{Lem kernels of x and overline x}, it suffices to show that (a) and (b) hold for some $\chi= x + \overline{x} \in U $.

Let $Y^1$ be the Young wall in $\mathcal{Y}^1(\Lambda_k)$ corresponding to $X$ under the crystal isomorphism
$ \mathcal{Y}^1(\Lambda_k) \cong \mathbb{B}(\Lambda_k), $
and $\mathfrak{b}_{ij}$ the $i$-th block from bottom in the $j$-th column of $Y^1$. By $\eqref{Eqn Conormal bundle of G-orbit of Y1} $, there exists $\chi= x + \overline{x} \in U $ such that
$$ x = g x(Y^1) g^{-1}  $$
for some $g \in G(\alpha)$. Then, by the equation $\eqref{Eqn ker of Y1}$, for $t \in \Z_{\ge 0}$, we have
\begin{align*}
 \ker x^t= \ker (g x(Y^1) g^{-1})^t = g (\ker (x(Y^1)^t)) = g \left(\Span_{\C} \{  v^{ \Color(\mathfrak{b}_{ij}) }_{ \mathfrak{o}(\mathfrak{b}_{ij}) } |\
\mathfrak{b}_{ij} \in Y^1,\ j < t \} \right).
\end{align*}
Let $y_t$ be the $t$-th column of $Y^1$ for $t \in \Z_{\ge0}$. Note that $y_t = \{ \mathfrak{b}_{it} \in Y^1|\  i \ge 1 \}$. Then, we have
\begin{align} \label{Eqn wt of y}
\wt(y_t) &= \sum_{ \mathfrak{b}_{it} \in y_t} \alpha_{\Color(\mathfrak{b}_{it} )} \nonumber   \\
&= \underline{\dim} \left( \Span_{\C} \{ v^{\Color(\mathfrak{b}_{it})}_{o(\mathfrak{b}_{it})} |\ \mathfrak{b}_{it} \in  y_t \} \right)  \nonumber \\
&= \underline{\dim} \left( \Span_{\C}\{ v^{\Color(\mathfrak{b}_{ij})}_{o(\mathfrak{b}_{ij})} |\
\mathfrak{b}_{ij} \in Y^1,\ j < t+1 \} \right)- \underline{\dim} \left( \Span_{\C}\{ v^{\Color(\mathfrak{b}_{ij})}_{o(\mathfrak{b}_{ij})} |\
\mathfrak{b}_{ij} \in Y^1,\ j < t \} \right) \nonumber   \\
&= \underline{\dim}(\ker(x(Y^1) )^{t+1}) - \underline{\dim}(\ker(x(Y^1) )^t) \nonumber  \\
&= \underline{\dim}(\ker x^{t+1}) - \underline{\dim}(\ker x^t) \nonumber \\
&= \underline{\dim}(\ker x^{t+1}/\ker x^t),
\end{align}
which implies that the height of $y_t$ is
$$ \Ht(\wt(y_t)) =  \dim ( \ker x^{t+1} / \ker x^t ).$$
Consequently, the assertion (a) follows from $\eqref{Eqn crystal iso from paths P1 to Young walls Y1}$ and $\eqref{Eqn iso 1 YWs to paths}$.

The remaining assertion (b) can be proved in the same manner.
\end{proof}




Combining the crystal isomorphisms $\eqref{Eqn crystal iso from paths P1 to Young walls Y1}$ and $\eqref{Eqn crystal iso from paths Pn to Young walls Yn}$:
\begin{align*}
\mathbf{Y}^1_k:\mathcal{P}^1(\Lambda_k) \longrightarrow \mathcal{Y}^1(\Lambda_k) \quad \text{ and } \quad
\mathbf{Y}^n_k:\mathcal{P}^n(\Lambda_k) \longrightarrow \mathcal{Y}^n(\Lambda_k)
\end{align*}
with $\eqref{Eqn Conormal bundle of G-orbit of Y1}$ and $\eqref{Eqn
Conormal bundle of G-orbit of Yn}$, we have the following
proposition, which, together with Theorem \ref{Thm paths of B1 and
Bn}, yields an explicit 1-1 correspondence between
$\mathbb{B}(\Lambda_k)$ and $\mathcal{P}^1(\Lambda_k)$ (resp.
$\mathcal{P}^n(\Lambda_k)$).

\begin{Prop}
Let
$$ \mathbf{q}_k^1:   \mathcal{P}^1(\Lambda_k) \longrightarrow \mathbb{B}(\Lambda_k)  \
\text{ {\rm (}resp. } \mathbf{q}_k^n:  \mathcal{P}^n(\Lambda_k)
\longrightarrow \mathbb{B}(\Lambda_k) \text{\rm )}$$ be the unique
crystal isomorphism given by Theorem \ref{Thm crystal iso of path
realization} and Theorem \ref{Thm saito geometric crystal}, and take
a $\Lambda_k$-path $\mathbf{p}^1 \in \mathcal{P}^1(\Lambda_k)$ {\rm
(}resp. $\mathbf{p}^n \in \mathcal{P}^n(\Lambda_k)${\rm )}. Let
$$ \alpha = \Lambda_k - \wt(\mathbf{p}^1)\ \text{ and }\ X_1 =  \mathbf{q}_k^1(\mathbf{p}^1)\quad \text{{\rm(}resp. } \beta = \Lambda_k - \wt(\mathbf{p}^n)\ \text{ and }\ X_n =  \mathbf{q}_k^n(\mathbf{p}^n) \text{\rm )}.$$
Then
\begin{enumerate}
\item[(a)]  $\iota_{\Lambda_k}(X_1)$ is the closure of the conormal bundle of the $G(\alpha)$-orbit of
$x(\mathbf{Y}^1_k(\mathbf{p}^1))$;
\item[(b)]  $\iota_{\Lambda_k}(X_n)$ is the closure of the conormal bundle of the $G(\beta)$-orbit of $\overline{x}(\mathbf{Y}^n_k(\mathbf{p}^n))$.
\end{enumerate}
\end{Prop}

Recall the fundamental isomorphism theorem of perfect crystals
$\eqref{Eqn fundamental thm of perfect crystals}$. From Theorem
\ref{Thm saito geometric crystal}, we have the following crystal
isomorphisms:
\begin{align*}
\psi_k^1: \mathbb{B}(\Lambda_k) \overset {\sim} \longrightarrow  \mathbb{B}(\Lambda_{k-1})\otimes B^1, \\
\psi_k^n: \mathbb{B}(\Lambda_k) \overset {\sim} \longrightarrow  \mathbb{B}(\Lambda_{k+1})\otimes B^n \\
\end{align*}
for $0 \le k \le n$. We would like to give a geometric interpretation to the crystal isomorphisms $\psi_i^1,\ \psi_i^n $ in terms of quiver varieties. To do that, we need a couple of lemmas.

Let $V$ be an $I$-graded vector space and $\chi$ an element of
$\Hom(V,V)$. If $W$ is a $\chi$-invariant $I$-graded subspace of
$V$, then $\chi$ can be viewed as an element in $\Hom(V/W,V/W)$
(resp. $\Hom(W,W)$), which is denoted by $\chi|_{V/W}$ (resp. $\chi
|_W$).

\begin{Lem} \label{Lem existance of the open set}
Let $x \in \bigoplus_{i\in I} \Hom(V_{i-1},V_i)$ for an $I$-graded vector space $V:= \bigoplus_{i\in I} V_i$, and set
$$ W := \ker x \ \text{ and }\ y := x |_{V/W}. $$
Take an element
$$ \overline{y} \in \bigoplus_{i\in I} \Hom(V_{i}/W_{i}, V_{i-1}/W_{i-1}) \quad \text{ with } \quad [y,\overline{y}]=0 ,$$
where $W_i$ is the $i$-subspace of $W$ for $i\in I$. Then there exists an element
$$\overline{x} \in \bigoplus_{i\in I} \Hom(V_{i}, V_{i-1}) $$
such that
$$ [x, \overline{x}]=0 \quad \text{ and } \quad \overline{x}|_{V/W} = \overline{y}. $$
\end{Lem}
\begin{proof}
Let $r = \dim V - \dim W$ and $s = \dim W$. Take an ordered basis
for $W$ and extend it to be an ordered basis for $V$ so that the
matrix representations of $x, y$ and $\overline{y}$ are given as
follows:
\begin{align*}
x = \left(
      \begin{array}{cc}
        0 & A \\
        0 & B \\
      \end{array}
    \right),\ \  y = B \ \ \text{ and } \ \ \overline{y}=C
\end{align*}
for some $r \times r$ matrices $B$ and $ C$, and $s \times r$ matrix $A$. Note that $ [B,C]=0$. Since
the matrix
$$ \left(
      \begin{array}{c}
        A  \\
        B  \\
      \end{array}
    \right) $$
has full rank, the following equation
\begin{align} \label{Eqn matrix equation}
 \left(
     \begin{array}{cc}
       X & Y \\
       0 & 0
     \end{array}
   \right) \cdot
\left(
      \begin{array}{cc}
       0 & A  \\
       0 & B  \\
      \end{array}
    \right)
=
\left(
      \begin{array}{cc}
       0 & AC   \\
       0 & 0  \\
      \end{array}
    \right)
\end{align}
has a solution. Since $x(V_i) \subset V_{i+1}$ and $ \tiny \left(
      \begin{array}{cc}
       0 & AC   \\
       0 & 0  \\
      \end{array}
    \right)$ maps $V_i$ to $V_i$,
we can choose an $s \times s$ matrix $X$ and an $s \times r$ matrix $Y$ such that
$ \tiny \left(
     \begin{array}{cc}
       X & Y \\
       0 & 0
     \end{array}
   \right) $ is a solution of the equation \eqref{Eqn matrix equation} and maps $V_i$ to $V_{i-1}$ for $i\in I$.
Let $$ \overline{x} = \left(
      \begin{array}{cc}
        X & Y \\
        0 & C \\
      \end{array}
    \right) .$$
By construction, we have
$$ \overline{x}(V_{i}) \subset V_{i-1}\ (i\in I), \quad \quad \overline{x}|_{V/W} = \overline{y},$$
and
$$ [x, \overline{x}] = x \overline{x} - \overline{x} x =
\left(
      \begin{array}{cc}
        0 & AC \\
        0 & BC \\
      \end{array}
    \right)
 -
\left(
      \begin{array}{cc}
        0 & XA+YB \\
        0 & CB \\
      \end{array}
    \right) = 0 .
 $$
\end{proof}

\begin{Lem} \label{Lem restriction of x}
Let $U$ be an open subset of $X_0 \in \Irr\Lambda(\alpha)$ as in
Lemma \ref{Lem kernels of x and overline x}. Set $$\beta =
\underline{\dim}(\ker x)\ (\text{resp. } \gamma =
\underline{\dim}(\ker \overline{x}))$$ for $\chi = x+\overline{x}
\in U$.
\begin{enumerate}
\item[(a)] There exists an irreducible component $X_0' \in \Irr \Lambda(\alpha - \beta)$ such that, for $\chi= x + \overline{x} \in U$,
\begin{align} \label{Eqn restriction of x}
 \phi \circ (\chi |_{V(\alpha)/ \ker x})\circ \phi^{-1} \in X_0',
\end{align}
where $\phi:V(\alpha)/\ker x \to V(\alpha - \beta)$ is an $I$-graded vector space isomorphism.
\item[(b)] There exists an irreducible component $X_0'' \in \Irr \Lambda(\alpha - \gamma)$ such that, for $\chi= x + \overline{x} \in U$,
\begin{align} \label{Eqn restriction of overlinex}
 \phi \circ (\chi |_{V(\alpha)/  \ker \overline{x}})\circ \phi^{-1} \in X_0'',
\end{align}
where $\phi:V(\alpha)/ \ker \overline{x} \to V(\alpha - \gamma)$ is an $I$-graded vector space isomorphism.
\end{enumerate}

\end{Lem}
\begin{proof}

Note that $\beta$ and $\gamma$ are well-defined by Lemma \ref{Lem kernels of x and overline x}. We first deal with the case (a). For an element $\chi = x + \overline{x} \in U$, let
$$\chi_{\phi} := \phi \circ (\chi|_{V(\alpha)/\ker x}) \circ \phi^{-1} \in \End(V(\alpha-\beta)),$$
where $\phi:V(\alpha)/ \ker x \to V(\alpha - \beta)$ is an $I$-graded vector space isomorphism.
Since $\chi \in \Lambda(\alpha)$, we have $$ \chi_{\phi} \in \Lambda(\alpha - \beta).$$
Take two elements $\chi = x+ \overline{x},\ \chi' = x'+ \overline{x}' \in U$, and choose two $I$-graded vector space isomorphisms $\phi:V(\alpha)/\ker x \to V(\alpha- \beta)$ and $\phi':V(\alpha)/\ker x' \to V(\alpha- \beta)$. From the properties of $U$ described in the proof of Lemma \ref{Lem kernels of x and overline x}, we have
$$ x = gx'g^{-1}  $$
for some $g \in G(\alpha)$, which yields that $\pi_{\Omega}(\chi_{\phi})$ and $\pi_{{\Omega}}(\chi'_{\phi'})$
are in the same $G(\alpha-\beta)$-orbit. Therefore, there exists an irreducible component $X_0' \in \Irr \Lambda(\alpha - \beta)$ such that
$$ \chi_\phi,\ {\chi'}_{\phi'}  \in X_0' .$$
Since $\chi, \chi'$ are arbitrary, our assertion follows.

The remaining case (b) can be proved in the same manner.
\end{proof}

\begin{Thm} \label{Thm fundamental thm B1 and Bn}
Let $X_0 = \iota_{\Lambda_k}(X)$ for an irreducible component $X \in \Irr \Lambda(\Lambda_k, \alpha)$. Set
$$d = \dim (\ker x)\ \text{ and }\  \beta = \underline{\dim}(\ker x) \ ( \text{resp. }
d' = \dim (\ker \overline{x})\ \text{ and }\  \gamma = \underline{\dim}(\ker \overline{x}))$$
for a generic point $\chi = x+\overline{x} \in X_0$.

\begin{enumerate}
\item [(a)] There exists a unique irreducible component $X' \in \Irr \Lambda(\Lambda_{k-1}, \alpha - \beta)$ satisfying the following conditions:
\begin{enumerate}
\item[(i)] there is an open subset $U \subset X_0$ such that, for $\chi= x+ \overline{x} \in U$,
$$ \phi \circ ( \chi|_{V(\alpha)/\ker x}) \circ \phi^{-1} \in \iota_{\Lambda_{k-1}}(X'), $$
where $\phi:V(\alpha)/\ker x \to V(\alpha - \beta)$ is an $I$-graded vector space isomorphism,
\item[(ii)] there is an open subset $U' \subset \iota_{\Lambda_{k-1}}(X')$ such that any element $\chi' \in U'$ can be written as
$$ \chi' = \phi \circ ( \chi|_{V(\alpha)/\ker x}) \circ \phi^{-1}, $$
for some $\chi = x + \overline{x} \in X_0$ and some $I$-graded vector space isomorphism $\phi:V(\alpha)/\ker x \to V(\alpha - \beta)$,
\item[(iii)] moreover, we have
 $$ \psi_k^1(X) = X' \otimes \mathbf{b}_a \quad \text{ and }\quad \wt(\mathbf{b}_a) = \Lambda_k - \Lambda_{k-1} - \cl(\beta) , $$
where $a \equiv d+k\ ( \text{mod}\ n+1)$.
\end{enumerate}

\item [(b)] There exists a unique irreducible component $X'' \in \Irr \Lambda(\Lambda_{k+1}, \alpha - \gamma)$ satisfying the following conditions:
\begin{enumerate}
\item[(i)] there is an open subset $U \subset X_0$ such that, for $\chi= x+ \overline{x} \in U$,
$$ \phi \circ ( \chi|_{V(\alpha)/\ker \overline{x}}) \circ \phi^{-1} \in \iota_{\Lambda_{k+1}}(X''), $$
where $\phi:V(\alpha)/\ker \overline{x} \to V(\alpha - \gamma)$ is an $I$-graded vector space isomorphism,
\item[(ii)] there is an open subset $U'' \subset \iota_{\Lambda_{k+1}}(X'') $ such that any element $\chi'' \in U''$ can be written as
$$ \chi'' = \phi \circ ( \chi|_{V(\alpha)/\ker \overline{x}}) \circ \phi^{-1}, $$
for some $\chi = x + \overline{x} \in X_0$ and some $I$-graded vector space isomorphism $\phi:V(\alpha)/\ker \overline{x} \to V(\alpha - \gamma)$,
\item[(iii)] moreover, we have
$$  \psi_k^n(X) = X'' \otimes \overline{\mathbf{b}}_b \quad \text{ and }\quad \wt(\overline{\mathbf{b}}_b) = \Lambda_k - \Lambda_{k+1} - \cl(\gamma ), $$
where  $b \equiv 1-d'+k\ ( \text{mod}\ n+1)$.
\end{enumerate}

\end{enumerate}

\end{Thm}

\begin{proof}
We first deal with the case (a) of the crystal isomorphism $\psi^1_k:\mathbb{B}(\Lambda_k) \to \mathbb{B}(\Lambda_{k-1}) \otimes B^1$. Let $Y$ be the Young wall in $\mathcal{Y}^1(\Lambda_k)$ corresponding to $X$ and $Y'$ the Young wall obtained by removing the $0$-th column from $Y$. Then $Y'$ can be viewed as an element in $\mathcal{Y}^1(\Lambda_{k-1})$. Take the irreducible component $X'$ in $ \mathbb{B}(\Lambda_{k-1})$ corresponding to $Y'$. By Theorem \ref{Thm paths of B1 and Bn} and $\eqref{Eqn wt of y}$, we have
$$X' \in \Irr\Lambda(\Lambda_{k-1}, \alpha - \beta)\quad \text{ and } \quad
 \psi_k^1(X) = X' \otimes \mathbf{b}_a ,$$
where $a \equiv d+k\ ( \text{mod}\ n+1)$ and $\wt(\mathbf{b}_a) = \Lambda_k - \Lambda_{k-1} - \cl(\beta)$.

Let $U$ be an open subset of $X_0$ as in Lemma \ref{Lem restriction of x}, and take an element $\chi = x + \overline{x} \in U$. Since $x(Y)|_{V(\alpha)/ \ker(x(Y))}$ is naturally identified with $x(Y')$ and $x$ is contained in the $G(\alpha)$-orbit of $x(Y)$, by Lemma \ref{Lem restriction of x}, we have
$$ \phi \circ(\chi|_{V(\alpha)/\ker x}) \circ \phi^{-1} \in \iota_{\Lambda_{k-1}}(X') $$
for an $I$-graded vector space isomorphism $\phi: V(\alpha)/\ker x \to V(\alpha-\beta)$.

Take an element $\chi' = x' + \overline{x}'$ in an open subset $U'$ of $\iota_{\Lambda_{k-1}}(X')$ given as in Lemma \ref{Lem kernels of x and overline x}. Then $x'$ can be written as
$$  x'  = g x(Y') g^{-1} $$
for some $g \in G(\alpha- \beta)$, which yields that there is $x$ in the $G(\alpha)$-orbit of $x(Y)$ such that
$$ \phi \circ (x|_{V(\alpha)/\ker x}) \circ \phi^{-1} = x'$$
for some $I$-graded vector space isomorphism $\phi: V(\alpha)/\ker x \to V(\alpha - \beta)$. The assertion (ii) follows from Lemma \ref{Lem existance of the open set}.

The remaining case (b), $\psi^n_k:\mathbb{B}(\Lambda_k) \to \mathbb{B}(\Lambda_{k+1}) \otimes B^n$, can be proved in the same manner.
\end{proof}

\vskip 3mm

\begin{Exa} \label{Exa B1 and Bn}
We use the same notations as in Example \ref{Exa the points from
Young walls}. Let $X_0 = \iota_{\Lambda_0}(X) \in \Irr
\Lambda(\alpha)$. By Theorem \ref{Thm Lusztig conormal bundle}, it
suffices to consider a generic point in the fiber
$\pi_{\overline{\Omega}}^{-1}(\overline{x}(Y^n)) \subset X_0$. By
\cite[Section 12.8, Proposition 15.5]{L91}, we have
\begin{align*}
\pi_{\overline{\Omega}}^{-1}(\overline{x}(Y^n)) &= \{ \chi \in X_0|\ \pi_{\overline{\Omega}}(\chi) = \overline{x}(Y^n) \} \\
&= \{ x + \overline{x}(Y^n) |\ x \in E_{\Omega}(\alpha),\ [x, \overline{x}(Y^n)]=0 \} \\
&= \{ x + \overline{x}(Y^n)\ |\ x = x(a_1,\ldots,a_{13}),\ a_1, \ldots, a_{13} \in \C \}.
\end{align*}
Here, $$x(a_1,\ldots,a_{13}) = \left(
                                        \begin{array}{cccc}
                                          0 & 0 & 0 & x_0 \\
                                          x_1 & 0 & 0 & 0 \\
                                          0 & x_2 & 0 & 0 \\
                                          0 & 0 & x_3 & 0 \\
                                        \end{array}
                                      \right)\in E_{\Omega}(\alpha), $$

\begin{align*}
 x_0 &= \left(
        \begin{array}{ccc}
          a_1 & a_2 & a_3  \\
          0 & 0 & 0  \\
          a_4 & 0 & a_5  \\
          a_6 & 0 & 0  \\
        \end{array}
      \right), \qquad \qquad \ \
x_1 = \left(
\begin{array}{cccc}
  0 & a_1 & a_2 & a_3 \\
  0 & a_4 & 0 & a_5 \\
  0 & a_6 & 0 & 0 \\
  0 & a_{10} & 0 & a_7 \\
\end{array}
\right) , \\
x_2 &= \left(
        \begin{array}{cccc}
          0 & a_2 & a_3 & 0 \\
          0 & 0 & a_5 & 0 \\
          a_6 & a_8 & a_{11} & a_9 \\
          0 & 0 & a_7 & 0 \\
        \end{array}
      \right), \qquad
x_3 = \left(
\begin{array}{cccc}
  0 & a_2 & 0 & 0 \\
  a_4 & a_{12} & a_5 & a_{13} \\
  a_6 & a_8 & 0 & a_9 \\
\end{array}
\right)
\end{align*}
for $a_1,\ldots, a_{13} \in \C$. Let
$$x := x(a_1,\ldots, a_{13}),\qquad  \overline{x} := \overline{x}(Y^n),$$
and consider $a_1, \ldots, a_{13}$ as indeterminates.
 Then we have
\begin{align*}
\dim(\ker x^k) = \left\{
                   \begin{array}{ll}
                     0 & \hbox{if }k=0,  \\
                     2+k & \hbox{if } 1 \le k \le 12, \\
                     15 & \hbox{otherwise},
                   \end{array}
                 \right.  \quad \text{and} \quad
\dim(\ker \overline{x}^k) = \left\{
                              \begin{array}{ll}
                                0 & \hbox{if } k=0, \\
                                4 & \hbox{if } k=1, \\
                                8 & \hbox{if } k=2,\\
                                11 & \hbox{if } k=3, \\
                                14 & \hbox{if } k=4,\\
                                15 & \hbox{otherwise}.
                              \end{array}
                            \right.
\end{align*}
Hence we obtain
\begin{align*}
\mathbf{p}^1_0(X) & = ( \ldots, \mathbf{b}_3, \mathbf{b}_1,\mathbf{b}_2, \mathbf{b}_3,
\mathbf{b}_0, \mathbf{b}_1,\mathbf{b}_2, \mathbf{b}_3,\mathbf{b}_0, \mathbf{b}_1,\mathbf{b}_2, \mathbf{b}_3,
\mathbf{b}_0, \mathbf{b}_3), \\
\mathbf{p}^n_0(X) & = (\ldots, \overline{\mathbf{b}}_3, \overline{\mathbf{b}}_2, \overline{\mathbf{b}}_0, \overline{\mathbf{b}}_1, \overline{\mathbf{b}}_0, \overline{\mathbf{b}}_2, \overline{\mathbf{b}}_1 ).
\end{align*}
\end{Exa}

\vskip 3em

\section{Quiver Varieties and Adjoint Crystals}

In this section, we will prove the main theorem of this paper,
Theorem \ref{Thm main thm}, which shows that there exists an
explicit crystal isomorphism between the geometric realization
$\mathbb{B}(\Lambda_0)$ and the path realization $\mathcal{P}^{\rm
ad}(\Lambda_0)$ of $B(\Lambda_0)$ arising from the adjoint crystal
$\adjoint$. By Theorem \ref{Thm crystal iso of path realization} and
Theorem \ref{Thm saito geometric crystal}, we have the crystal
isomorphism
$$\mathbf{p}^{\rm ad} : \mathbb{B}(\Lambda_0) \longrightarrow \mathcal{P}^{\rm ad}(\Lambda_0).$$

Let $\alpha \in Q^+$ and let $X$ be an irreducible component of
$\Irr \Lambda(\Lambda_0, \alpha)$. For a generic point $\chi = x +
\overline{x} \in \iota_{\Lambda_0}(X)$, we will give an explicit
description of the $\Lambda_0$-path $\mathbf{p}^{\rm ad}(X)$ in
terms of dimension vectors of $\ker (x\overline{x})^{k+1}/ \ker
(x\overline{x})^k $ for $k \ge0$.

Let $\alpha, \beta\in Q^+$ with $\beta \le \alpha$. Consider the
diagram given in \cite{L91}:
\begin{align} \label{Eqn the diagram of Lusztig quivers}
\Lambda(\alpha - \beta) \buildrel \pi \over \longleftarrow \Lambda(\alpha - \beta)\times \Lambda(\beta)
\buildrel p_1 \over \longleftarrow F' \buildrel p_2 \over \longrightarrow F'' \buildrel p_3 \over \longrightarrow
\Lambda(\alpha),
\end{align}
where $F''$ is the variety of all pairs $(\chi, W)$ such that
\begin{enumerate}
\item[(a)] $\chi \in \Lambda(\alpha)$,
\item[(b)]  $W$ is a $\chi$-stable subspace of $V(\alpha)$ with $\underline{\dim}W = \beta$,
 \end{enumerate}
and $F'$ is the variety of all quadruples $(\chi, W, f, g)$ such that
 \begin{enumerate}
 \item[(a)]  $(\chi, W) \in F''$,
 \item[(b)]  $f=(f_i)_{i \in I},g=(g_i)_{i \in I}$ give an exact sequence
$$ 0 \longrightarrow V_i(\beta) \buildrel f_i \over \longrightarrow V_i(\alpha) \buildrel g_i \over \longrightarrow V_i(\alpha - \beta) \longrightarrow 0 \quad (i\in I)$$
such that $\im f = W$.
\end{enumerate}
Then we have
$$p_1(\chi, W, f,g) = ( \tilde{g} \circ(\chi|_{V(\alpha)/W}) \circ \tilde{g}^{-1},\  f^{-1} \circ (\chi|_W) \circ f ),$$
where $\tilde{g}:V(\alpha)/W \to V(\alpha - \beta)$ is the $I$-graded vector space isomorphism induced by $g$,
$$p_2(\chi, W, f,g) = (\chi, W),\quad p_3(\chi, W) = \chi,$$
and $\pi$ is the natural first projection.
Note that $p_2$ is a $G(\alpha-\beta) \times G(\beta)$-principal bundle and an open map.

Let $U$ be an open subset of $X_0 \in \Irr\Lambda(\alpha)$ as in Lemma \ref{Lem kernels of x and overline x}, and $\beta = \underline{\dim}(\ker x)$ for $\chi = x + \overline{x} \in U$. Define the map $\imath:U \to F''$ by
$$\imath(\chi) = (\chi,\ \ker x ) $$
for $\chi = x + \overline{x} \in U$. Note that $p_3 \circ \imath = {\rm id}|_{U}$.
By Lemma \ref{Lem restriction of x}, there exists an irreducible component $X_0' \in \Irr\Lambda(\alpha - \beta)$ such that, for any $\chi = x + \overline{x} \in U$,
$$ \phi \circ (\chi|_{V(\alpha)/\ker x})\circ \phi^{-1} \in X_0',$$
where $\phi:V(\alpha)/\ker x \to V(\alpha - \beta) $ is an $I$-graded vector space isomorphism.
Given an open subset $U' \subset \Lambda(\alpha - \beta)$ with $U' \cap X_0' \ne \emptyset$, by Lemma \ref{Lem existance of the open set},
$$\tilde{U} := \imath^{-1} \circ p_2 \circ p_1^{-1} \circ \pi^{-1}(U') $$
is a nonempty open subset of $X_0$.
Therefore, given an open subset $U' \subset X_0'$, there exists an open subset $\tilde{U} \subset X_0$ such that,
for any element $\chi=x+\overline{x} \in \tilde{U}$,
$$ \phi \circ ( \chi|_{V(\alpha)/\ker x}) \circ \phi^{-1} \in U'$$
for some $I$-graded vector space isomorphism $\phi:V(\alpha)/\ker x \to V(\alpha - \beta)$.

In the same manner, let $ \gamma = \underline{\dim} (\ker \overline{x})$, and consider the diagram
\begin{align*}
\Lambda(\alpha - \gamma) \buildrel \pi \over \longleftarrow \Lambda(\alpha - \gamma)\times \Lambda(\gamma)
\buildrel p_1 \over \longleftarrow F' \buildrel p_2 \over \longrightarrow F'' \buildrel p_3 \over \longrightarrow
\Lambda(\alpha).
\end{align*}
Define the map $\overline{\imath}: U \to F''$ by $$\overline{
\imath }(\chi) = (\chi, \ker \overline{x})$$ for
$\chi=x+\overline{x} \in U$, and let $X_0''$ be an irreducible component as
in Lemma \ref{Lem restriction of x}. Then one can deduce that,
given an open subset $U' \subset X_0''$, there exists an open
subset $\tilde{U} \subset X_0$ such that, for any element
$\chi=x+\overline{x} \in \tilde{U}$,
$$ \phi \circ ( \chi|_{V(\alpha)/\ker \overline{x}}) \circ \phi^{-1} \in U'$$
for some $I$-graded vector space isomorphism $\phi:V(\alpha)/\ker \overline{x} \to V(\alpha - \beta)$.
Consequently, we have the following lemma.

\begin{Lem} \label{Lem pull-back an open set}
With the same notations as in Lemma \ref{Lem restriction of x}, we have the following.
\begin{enumerate}
\item[(a)] Given an open subset $U' \subset X_0'$, there exists an open subset $\tilde{U} \subset X_0$ such that, for any element  $\chi=x+\overline{x} \in \tilde{U}$,
\begin{align}\label{Eqn pull-back of x}
 \phi \circ ( \chi|_{V(\alpha)/\ker x}) \circ \phi^{-1} \in U'
\end{align}
for some $I$-graded vector space isomorphism $\phi:V(\alpha)/\ker x \to V(\alpha - \beta)$.
\item[(b)] Given an open subset $U'' \subset X_0''$, there exists an open subset $\tilde{U} \subset X_0$ such that, for any element  $\chi=x+\overline{x} \in \tilde{U}$,
\begin{align}\label{Eqn pull-back of overline x}
 \phi \circ ( \chi|_{V(\alpha)/ \ker \overline{x}}) \circ \phi^{-1} \in U''
\end{align}
for some $I$-graded vector space isomorphism $\phi:V(\alpha)/ \ker \overline{x} \to V(\alpha - \gamma)$.
\end{enumerate}
\end{Lem}

Combining Lemma \ref{Lem pull-back an open set} with Lemma \ref{Lem restriction of x}, we have the following lemma.

\begin{Lem} \label{Lem kernels of overline-x x}
Let $\alpha \in Q^+$. For each $X_0 \in  \Irr \Lambda(\alpha)$, there exists an open subset $U \subset X_0$ such that
\begin{align} \label{Eqn x-stable of x overline x}
\ker(x \overline{x})^k \cong \ker(x' \overline{x}')^k\ \text{ and }\  \ker x(x \overline{x})^k \cong \ker x'(x' \overline{x}')^k
\end{align} \label{Eqn kernels of overline-x x}
for any $\chi = x+ \overline{x}, \chi' = x'+ \overline{x}' \in U$ and $k\in \Z_{\ge 0}$.
\end{Lem}
\begin{proof}
Since the case that $\Ht (\alpha) = 0$ is trivial, we may assume $\Ht (\alpha) > 0$.
Let $U_0$ be an open subset of $X_0$ as in Lemma \ref{Lem kernels of x and overline x}, and $\beta = \underline{\dim} (\ker x)$ for $\chi = x + \overline{x} \in U_0$. Take the irreducible component $X_0' \in \Irr \Lambda(\alpha-\beta)$ given in Lemma \ref{Lem restriction of x}, and choose an open subset $U_0'$ of $X_0'$ satisfying the conditions of Lemma \ref{Lem kernels of x and overline x}. Let $\gamma = \underline{\dim} (\ker \overline{x})$ for $\chi = x + \overline{x} \in U_0'$. By Lemma \ref{Lem pull-back an open set}, there exists an open subset $\tilde{U}_0 \subset X_0$ such that, for any element  $\chi=x+\overline{x} \in \tilde{U}_0$,
$$ \phi \circ ( \chi|_{V(\alpha)/\ker x}) \circ \phi^{-1} \in U_0'$$
for some $I$-graded vector space isomorphism $\phi:V(\alpha)/\ker x \to V(\alpha - \beta)$. Set
$ \hat{U}_0 = U_0 \cap \tilde{U}_0 .$ Then, for any $\chi = x + \overline{x} \in \hat{U}_0$, we have
\begin{align*}
\underline{\dim} \ker x \overline{x} &=  \underline{\dim} (\ker x) + \underline{\dim} (\ker x \overline{x}/\ker x) \\
&= \underline{\dim} (\ker x) + \underline{\dim}\{ v \in V(\alpha)/\ker x\ |\ \overline{x}|_{V(\alpha)/\ker x}(v) = 0 \} \\
&= \underline{\dim} (\ker x) + \underline{\dim} (\ker \overline{x}|_{V(\alpha)/\ker x}) \\
&= \beta + \gamma.
\end{align*}

Let us take the irreducible component $X_0'' \in \Irr \Lambda(\alpha - \beta - \gamma)$ associated with $X_0'$, which is given as in Lemma \ref{Lem restriction of x}. By the induction hypothesis, there exists an open subset $U'' \subset X_0''$ satisfying $\eqref{Eqn x-stable of x overline x}$. Applying Lemma \ref{Lem pull-back an open set} to $X_0'', X_0'$ and $X_0$, there exists an open subset $\tilde{U} \subset X_0$ such that, for any $\chi = x + \overline{x} \in \hat{U}$,
\begin{align}\label{Eqn restriction to V / ker x overline x}
 \phi \circ ( \chi|_{V(\alpha)/\ker x\overline{x}}) \circ \phi^{-1} \in U''
\end{align}
for some $I$-graded vector space isomorphism $\phi:V(\alpha)/\ker x\overline{x} \to V(\alpha - \beta - \gamma)$.
Let $U = \tilde{U} \cap \hat{U}_0$, then, by construction, $U$ holds the condition $\eqref{Eqn x-stable of x overline x}$.
\end{proof}

An element $\chi \in X_0$ in the open subset $U \subset X_0$
satisfying $\eqref{Eqn constant kernel - generic properties of X}$
and $\eqref{Eqn x-stable of x overline x}$ is called a {\it generic
point}. Note that, for $\chi = x + \overline{x} $ in an irreducible
component $ X_0 \in \Irr \Lambda(\alpha)$, since $[x,
\overline{x}]=0$, we have
$$ \dim( \ker (x|_{\ker (x\overline{x})^{k+1} / \ker (x\overline{x})^k})) = \dim(\ker x(x\overline{x})^k) - \dim(\ker (x\overline{x})^k) ,$$
where $ x|_{\ker (x\overline{x})^{k+1} / \ker (x\overline{x})^k}$ is the linear map in $\End(\ker (x\overline{x})^{k+1} / \ker (x\overline{x})^k)$ induced by $x$.
Thanks to Lemma \ref{Lem kernels of overline-x x}, one can talk about
\begin{align*}
& \underline{\dim} \ker (x\overline{x})^k, \quad \underline{\dim} \ker (x\overline{x})^{k+1} / \ker (x\overline{x})^k
\ \text{ and }\ \dim( \ker (x|_{\ker (x\overline{x})^{k+1} / \ker (x\overline{x})^k}))
\end{align*}
for a generic point $\chi = x + \overline{x} $ in an irreducible component $ X_0 \in \Irr \Lambda(\alpha)$ and $k \in \Z_{\ge 0}$.

Finally, we are ready to state the main theorem in this paper.

\begin{Thm} \label{Thm main thm}
Let
$$\mathbf{p}^{\rm ad}:\mathbb{B}(\Lambda_0) \longrightarrow \mathcal{P}^{\rm ad}(\Lambda_0)$$
be the unique crystal isomorphism given by Theorem \ref{Thm crystal
iso of path realization} and Theorem \ref{Thm saito geometric
crystal}, and take an irreducible component $X \in
\mathbb{B}(\Lambda_0)$. For a generic point $\chi = x + \overline{x}
\in \iota_{\Lambda_0}(X)$ and $k \in \Z_{\ge 0}$, let
\begin{align*}
\theta_k  &= \underline{\dim} ( \ker (x \overline{x})^{k+1} /\ker (x \overline{x})^k ), \quad  \\
c_k  &\equiv \dim(\ker (x|_{\ker (x \overline{x})^{k+1} /\ker (x \overline{x})^k }) )  \ ( \text{mod}\ n+1),
\end{align*}
where $0 \le c_k \le n $.
Then we have
\begin{align}\label{Eqn main theorem}
 \mathbf{p}^{\rm ad}(X) = (\ldots, p_k, \ldots, p_1, p_0),
\end{align}
where
$$ p_k = \left\{
           \begin{array}{ll}
             \mathsf{b}_{-\cl(\theta_k)} & \hbox{if $\cl(\theta_k) \ne 0$,} \\
             h_{c_k} & \hbox{if $\cl(\theta_k) = 0$ and $c_k \ne 0$,} \\
             \emptyset & \hbox{otherwise.}
           \end{array}
         \right.
 $$
\end{Thm}

\begin{proof}
Let $X_0 = \iota_{\Lambda_0}(X)$ with $\wt(X) = \Lambda_0 - \alpha$ for some $\alpha \in Q^+$. We will use induction on $\Ht(\alpha)$. Since the case that $\Ht(\alpha) = 0$ is trivial, we may assume $\alpha \ne 0$. Note that
\begin{align*}
\theta_k &= \underline{\dim} (\ker(x \overline{x})^{k+1}) - \underline{\dim} (\ker(x \overline{x})^{k}), \\
c_k &\equiv \dim(\ker x(x \overline{x})^{k}) -  \dim(\ker (x \overline{x})^{k})  \ ( \text{mod}\ n+1)
\end{align*}
for a generic point $\chi = x + \overline{x} \in X_0$ and $k \in \Z_{\ge0}$.

Let $\beta = \underline{\dim}(\ker x) $ for a generic point $\chi = x + \overline{x} \in X_0$, and choose the irreducible component $X_0' \in \Irr \Lambda(\alpha - \beta)$ associated with $X_0$ as in Lemma \ref{Lem restriction of x}. Similarly, let $ \gamma = \underline{\dim} (\ker \overline{x})$ for a generic point $\chi = x + \overline{x}\in X_0'$, and take the irreducible component $X_0'' \in \Irr \Lambda(\alpha - \beta - \gamma)$ associated with $X_0'$ as in Lemma \ref{Lem restriction of x}.
By Theorem \ref{Thm fundamental thm B1 and Bn}, we have
$$ \psi_{0}^1(X_0) = X_0' \otimes \mathbf{b}_a \quad \text{and} \quad \psi_{n}^n(X_0') = X_0'' \otimes \overline{\mathbf{b}}_b $$
for some $\mathbf{b}_a \in B^1,\ \overline{\mathbf{b}}_b \in B^n $.
From $\eqref{Eqn fundamental thm of perfect crystals}$ and Theorem \ref{Thm K-S geometric crystal}, we have
$$\psi_0^{\rm ad}: \mathbb{B}(\Lambda_0) \overset {\sim} \longrightarrow \mathbb{B}(\Lambda_{0})\otimes \adjoint.$$
Then, it follows from the crystal isomorphism $\eqref{Eqn isomorphism of perfect crystals}$ and Theorem \ref{Thm fundamental thm B1 and Bn} that
$$\psi_0^{\rm ad}(X_0) = X_0'' \otimes \mathfrak{p}^{\rm ad}( \overline{\mathbf{b}}_b \otimes \mathbf{b}_a ).$$
By the induction hypothesis, there is an open subset $U'' \subset X_0''$ satisfying $\eqref{Eqn main theorem}$. By Lemma \ref{Lem pull-back an open set} and $\eqref{Eqn restriction to V / ker x overline x}$, there is an open subset $\tilde{U} \subset X_0$ such that, for any $\chi = x + \overline{x} \in \tilde{U}$,
$$ \phi \circ ( \chi|_{V(\alpha)/\ker x\overline{x}}) \circ \phi^{-1} \in U''$$
for some isomorphism $\phi:V(\alpha)/\ker x\overline{x} \to V(\alpha - \beta - \gamma)$.

On the other hand, by Lemma \ref{Lem kernels of overline-x x}, there exists an open subset $\hat{U} \subset X_0$ satisfying $\eqref{Eqn x-stable of x overline x}$. Set
$$ U = \hat{U} \cap \tilde{U} $$
and choose an element $\chi = x + \overline{x} \in U$.
Suppose that $\wt(\overline{\mathbf{b}}_b \otimes \mathbf{b}_a ) \ne 0$. Then, by Theorem \ref{Thm fundamental thm B1 and Bn} and $\eqref{Eqn isomorphism of perfect crystals}$, since
\begin{align*}
 \wt(\overline{\mathbf{b}}_b \otimes \mathbf{b}_a ) &= \wt( \overline{\mathbf{b}}_b) + \wt(\mathbf{b}_a ) \\
&=\Lambda_n-\Lambda_0 - \cl(\gamma) + \Lambda_0-\Lambda_n  -\cl(\beta) \\
& = - \cl(\beta + \gamma) \\
&= -\cl(\underline{\dim}(\ker x \overline{x}))\\
&= -\cl(\theta_0),
\end{align*}
we obtain
$$\mathfrak{p}^{\rm ad}(\overline{\mathbf{b}}_b \otimes \mathbf{b}_a )
= \mathsf{b}_{\wt(\overline{\mathbf{b}}_b \otimes \mathbf{b}_a )} = \mathsf{b}_{-\cl(\theta_0)}.$$
Suppose $\wt(\overline{\mathbf{b}}_b \otimes \mathbf{b}_a ) = 0$ and $a \ne n+1$. Then, by Theorem \ref{Thm fundamental thm B1 and Bn} and $\eqref{Eqn isomorphism of perfect crystals}$, we have
\begin{align*}
\mathfrak{p}^{\rm ad}(\overline{\mathbf{b}}_b \otimes \mathbf{b}_a  ) = h_{a}
\end{align*}
and  $ a \equiv \dim x \ ( \text{mod}\ n+1)$, which implies that $a = c_0$. In the same manner, if $\wt(\overline{\mathbf{b}}_b \otimes \mathbf{b}_a ) = 0$ and $a = n+1$, we have
\begin{align*}
\mathfrak{p}^{\rm ad}(\overline{\mathbf{b}}_b \otimes \mathbf{b}_a  ) = \emptyset.
\end{align*}
Since, for an arbitrary isomorphism $\phi:V(\alpha)/\ker x\overline{x} \to V(\alpha - \beta - \gamma)$,
\begin{align*}
\underline{\dim}(\ker(x \overline{x})^{k+1}) &= \underline{\dim}(\ker(x \overline{x})^{k+1}/ \ker(x \overline{x})) + \underline{\dim}(\ker(x \overline{x})) \\
&= \underline{\dim}(\ker(x \overline{x})|_{V(\alpha)/ \ker(x \overline{x})})^k + \underline{\dim}(\ker(x \overline{x})) \\
&= \underline{\dim}\ker(\phi \circ (x \overline{x})|_{V(\alpha)/ \ker(x \overline{x})}\circ \phi^{-1})^k + \underline{\dim}(\ker(x \overline{x})) \\
&= \underline{\dim}\ker((\phi \circ x|_{V(\alpha)/ \ker(x \overline{x})} \circ \phi^{-1})( \phi \circ \overline{x}|_{V(\alpha)/ \ker(x \overline{x})}\circ \phi^{-1}))^k + \underline{\dim}(\ker(x \overline{x})),
\end{align*}
our assertion follows from a standard induction argument.
\end{proof}

The following corollary, which is an immediate consequence of
Theorem \ref{Thm fundamental thm B1 and Bn} and Theorem \ref{Thm
main thm}, can be regarded as a geometric interpretation of the
fundamental isomorphism theorem for perfect crystals
$$
\psi_0^{\rm ad}: \mathbb{B}(\Lambda_0) \overset {\sim} \longrightarrow \mathbb{B}(\Lambda_{0})\otimes \adjoint.
$$

\begin{Cor} \label{Cor main thm}
Let $X_0 = \iota_{\Lambda_0}(X)$ for some $X \in \Irr\Lambda(\Lambda_0, \alpha)$. For a generic point $\chi = x + \overline{x} \in X_0$, set
$$ \theta = \underline{\dim}(\ker(x \overline{x}))\  \text{ and }\  c = \dim (\ker x) .$$
Then there exists a unique irreducible component $X' \in \Irr  \Lambda(\Lambda_0, \alpha - \theta)$ satisfying the following conditions:
\begin{enumerate}
\item[(a)] there is an open subset $U \subset X_0$ such that, for $\chi = x + \overline{x} \in U$,
$$ \phi \circ (\chi|_{V(\alpha)/ \ker x \overline{x}}) \circ \phi^{-1} \in \iota_{\Lambda_0}(X'), $$
where $\phi: V(\alpha-\theta) \to V(\alpha)/\ker{x \overline{x}}$ is an $I$-graded vector space isomorphism,
\item[(b)] there is an open subset $U' \subset \iota_{\Lambda_0}(X')$ such that any element $\chi' \in U'$ can be written as
$$ \chi' = \phi \circ (\chi |_{V(\alpha)/\ker x \overline{x}})\circ \phi^{-1} $$
for some $\chi = x + \overline{x} \in X_0$ and some $I$-graded vector space isomorphism $\phi:V(\alpha)/\ker x \overline{x} \to V(\alpha - \theta)$,
\item[(c)] moreover, we have
$$ \psi_{0}^{\rm ad}(X) = X' \otimes p, $$
where
$$ p = \left\{
         \begin{array}{ll}
           \mathsf{b}_{- \cl(\theta)} & \hbox{if $\cl(\theta) \ne 0$,} \\
           h_{c} & \hbox{if $\cl(\theta) = 0$ and $c \ne 0$,} \\
           \emptyset & \hbox{otherwise.}
         \end{array}
       \right.
$$
\end{enumerate}
\end{Cor}

\vskip 5em

\begin{Exa} We use the same notations as in Example \ref{Exa B1 and Bn}. Let
$W_i = \ker (x \overline{x})^i$
for $i \in \Z_{\ge0}$. Then we have
$$ \underline{\dim} (W_i) = \left\{
                                      \begin{array}{ll}
                                        0 & \hbox{if } i=0, \\
                                        \alpha_0 + 2\alpha_1 + 2\alpha_2 + \alpha_3 & \hbox{if } i=1, \\
                                        2\alpha_0 + 3\alpha_1 + 3\alpha_2 + 2\alpha_3 & \hbox{if } i=2, \\
                                        3\alpha_0 + 4\alpha_1 + 4\alpha_2 + 3\alpha_3 & \hbox{if } i=3, \\
                                        4\alpha_0 + 4\alpha_1 + 4\alpha_2 + 3\alpha_3 & \hbox{otherwise,}
                                      \end{array}
                                    \right.
   $$
and
$$ \dim(\ker(x|_{W_{i+1}/W_i}) ) = \dim( \ker x (x\overline{x})^i) - \dim(\ker (x\overline{x})^i) =
\left\{
  \begin{array}{ll}
    3 & \hbox{if } i=0, \\
    1 & \hbox{if }  i=1, \\
    1 & \hbox{if }  i=2, \\
    0 & \hbox{otherwise.}
  \end{array}
\right.
$$
By Theorem \ref{Thm main thm}, we have
$$ \mathbf{p}^{\rm ad}(X) = (\ldots, \emptyset, \emptyset, \mathsf{b}_{\alpha_1 + \alpha_2+\alpha_3}, h_1, h_1,
\mathsf{b}_{-\alpha_1 - \alpha_2}).$$
\end{Exa}

\vskip 5em

\bibliographystyle{amsalpha}

\end{document}